\documentclass{article}
\usepackage{makeidx}
\usepackage{amssymb}
\usepackage{amsmath}
\usepackage{amstext}
\usepackage{graphicx}
\usepackage{latexsym}
\usepackage{amsthm}
\newtheorem{theorem}{Theorem}

\newtheorem{corollary}[theorem]{Corollary}

\newtheorem{definition}[theorem]{Definition}
\newtheorem{example}[theorem]{Example}

\newtheorem{lemma}[theorem]{Lemma}

\newtheorem{proposition}[theorem]{Proposition}
\newtheorem{remark}[theorem]{Remark}

\newtheorem{prop}[theorem]{Proposition}
\newtheorem{desideratum}{Desideratum}
\input tcilatex
\begin{document}

\title{Generalized functions beyond distributions}
\author{Vieri Benci\thanks{
Dipartimento di Matematica, Universit\`{a} degli Studi di Pisa, Via F.
Buonarroti 1/c, 56127 Pisa, ITALY and Department of Mathematics, College of
Science, King Saud University, Riyadh, 11451, SAUDI ARABIA. e-mail: \texttt{%
benci@dma.unipi.it}} \and Lorenzo Luperi Baglini\thanks{%
University of Vienna, Faculty of Mathematics, Oskar-Morgenstern-Platz 1,
1090 Vienna, AUSTRIA, e-mail: \texttt{lorenzo.luperi.baglini@univie.ac.at},
supported by grant P25311-N25 of the Austrian Science Fund FWF.}}
\maketitle

\begin{abstract}
Ultrafunctions are a particular class of functions defined on a Non
Archimedean field $\mathbb{R}^{\ast }\supset \mathbb{R}$. They have been
introduced and studied in some previous works (\cite{ultra},\cite{belu2012},%
\cite{belu2013}). In this paper we introduce a modified notion of
ultrafunction and we discuss sistematically the properties that this
modification allows. In particular, we will concentrated on the definition
and the properties of the operators of derivation and integration of
ultrafunctions.

\end{abstract}

\tableofcontents

\section{Introduction}

In some recent papers the notion of ultrafunction has been introduced and
studied (\cite{ultra}, \cite{belu2012}, \cite{belu2013}). Ultrafunctions are
a particular class of functions defined on a Non Archimedean field $\mathbb{R%
}^{\ast }\supset \mathbb{R}$. We recall that a Non Archimedean field is an
ordered field which contains infinite and infinitesimal numbers. In general,
as we showed in our previous works, when working with ultrafunctions we
associate to any continuous function $f:\mathbb{R}^{N}\rightarrow \mathbb{R}$
an ultrafunction $\widetilde{f}:\left( \mathbb{R}^{\ast }\right)
^{N}\rightarrow \mathbb{R}^{\ast }$ which extends $f;$ more exactly, to any
vector space of functions $V(\Omega )\subseteq L^{2}(\Omega )\cap \mathcal{C}%
(\overline{\Omega })$ we associate a space of ultrafunctions $\widetilde{V}%
(\Omega ).\ $The spaces of ultrafunctions are much larger than the
corrispective spaces of functions, and have much more "compactness": these
two properties ensure that in the spaces of ultrafunctions we can find
solutions to functional equations which do not have any solutions among the
real functions or the distributions.

In \cite{belu2013} we studied the basic properties of ultrafunctions. One
property that is missing, in general, is the "locality": local changes to an
ultrafunction (namely, changing the value of an ultrafunction in a
neighborhood of a point) affects the ultrafunction globally (namely, they
may force to change the values of the ultrafunction in all the points). This
problem is related to the properties of a particular basis of the spaces of
ultrafunctions, called "Delta basis" (see \cite{belu2012}, \cite{belu2013}).
The elements of a Delta basis are called Delta ultrafunctions and, in some
precise sense, they are an analogue of the Delta distributions. More
precisely, given a point $a\in \mathbb{R}^{\ast },$ the Delta ultrafunction
centered in $a$ (denoted by $\delta _{a}(x)$) is the unique ultrafunction
such that, for every ultrafunction $u(x)$, we have%
\begin{equation*}
\int^{\ast }u(x)\delta _{a}(x)dx=u(a).\footnote{$\int^{\ast }:L^{1}(\mathbb{R%
})^{\ast }\rightarrow \mathbb{C}^{\ast }$ is an extension of the integral $%
\int :L^{1}(\mathbb{R})\rightarrow \mathbb{C}.$}
\end{equation*}%
It would be useful for applications to have an orthonormal Delta basis,
namely a Delta basis $\{\delta _{a}(x)\}_{a\in \Sigma }$ such that, for
every $a,b\in \Sigma ,$ $\int^{\ast }\delta _{a}(x)\delta _{b}(x)dx=\delta
_{a,b}$; unfortunately, this seems to be impossible.

The main aim of this paper is to show how to modify the constructions
exposed in \cite{belu2013} (that will be recalled) to avoid such unwanted
issues. We will show how to construct spaces of ultrafunctions that have
"good local properties" and that have Delta bases $\{\delta _{a}(x)\}_{a\in
\Sigma }$ that are "almost orthogonal" where, by saying that a Delta basis
is "almost orthogonal", we mean the following: for every $a,b\in \Sigma ,$
if $|a-b|$ is not infinitesimal$\footnote{%
We recall that an element $x$ of a Non Archimedean superreal ordered field $%
\mathbb{K}\supset \mathbb{R}$ is infinitesimal if $|x|<r$ for every $r\in 
\mathbb{R.}$}$ then $\int^{\ast }\delta _{a}(x)\delta _{b}(x)dx=0.$

We will also discuss a few other properties of ultrafunctions that were
missing in the previous approach but that hold in this new context.

The techniques on which the notion of ultrafunction is based are related to
Non Archimedean Mathematics (NAM) and to Nonstandard Analysis (NSA). In
particular, the most important notion that we use is that of $\Lambda $%
-limit (see \cite{ultra}, \cite{belu2012}, \cite{belu2013}). In this paper
this notion will be considered known; however, for sake of completeness, we
will recall its basic properties in the Appendix.

\subsection{Notations\label{not}}

If $X$ is a set then

\begin{itemize}
\item $\mathcal{P}\left( X\right) $ denotes the power set of $X$ and $%
\mathcal{P}_{fin}\left( X\right) $ denotes the family of finite subsets of $%
X;$

\item $\mathfrak{F}\left( X,Y\right) $ denotes the set of all functions from 
$X$ to $Y$ and $\mathfrak{F}\left( \mathbb{R}^{N}\right) =\mathfrak{F}\left( 
\mathbb{R}^{N},\mathbb{R}\right) .$
\end{itemize}

Let $\Omega $\ be a subset of $\mathbb{R}^{N}$: then

\begin{itemize}
\item $\mathcal{C}\left( \Omega \right) $ denotes the set of continuous
functions defined on $\Omega \subset \mathbb{R}^{N};$

\item $\mathcal{C}_{0}\left( \Omega \right) $ denotes the set of continuous
functions in $\mathcal{C}\left( \Omega \right) $ having compact support in $%
\Omega ;$

\item $\mathcal{C}^{k}\left( \Omega \right) $ denotes the set of functions
defined on $\Omega \subset \mathbb{R}^{N}$ which have continuous derivatives
up to the order $k;$

\item $\mathcal{C}_{0}^{k}\left( \Omega \right) $ denotes the set of
functions in $\mathcal{C}^{k}\left( \Omega \right) \ $having compact support;

\item $\mathcal{C}_{\sharp }^{1}\left( \mathbb{R}\right) $ denotes the set
of functions $f$ of class $\mathcal{C}^{1}\left( \Omega \right) \ $except
than on a discrete set $\Gamma \subset \mathbb{R}$ and such that, for any $%
\gamma \in \Gamma ,$ the limits 
\begin{equation*}
\underset{x\rightarrow \gamma ^{\pm }}{\lim }f(x)
\end{equation*}%
exist and are finite;

\item $\mathcal{D}\left( \Omega \right) $ denotes the set of the infinitely
differentiable functions with compact support defined on $\Omega \subset 
\mathbb{R}^{N};\ \mathcal{D}^{\prime }\left( \Omega \right) $ denotes the
topological dual of $\mathcal{D}\left( \Omega \right) $, namely the set of
distributions on $\Omega ;$

\item if $A\subset \mathbb{R}^{N}$ is a set, then $\chi _{A}$ denotes the
characteristic function of $A;$

\item for any $\xi \in \left( \mathbb{R}^{N}\right) ^{\ast },\rho \in 
\mathbb{R}^{\ast }$, we set $\mathfrak{B}_{\rho }(\xi )=\left\{ x\in \left( 
\mathbb{R}^{N}\right) ^{\ast }:\ |x-\xi |<\rho \right\} $;

\item $\mathfrak{supp}(f)=\overline{\left\{ x\in \left( \mathbb{R}%
^{N}\right) ^{\ast }:f(x)\neq 0\right\} };$

\item $\mathfrak{mon}(x)=\{y\in \left( \mathbb{R}^{N}\right) ^{\ast }:x\sim
y\};$

\item $\mathfrak{gal}(x)=\{y\in \left( \mathbb{R}^{N}\right) ^{\ast }:x\sim
_{f}y\};$

\item $\forall ^{a.e.}$ $x\in X$ means "for almost every $x\in X";$

\item if $a,b\in \mathbb{R}^{\ast },$ then

\begin{itemize}
\item $\left[ a,b\right] _{\mathbb{R}^{\ast }}=\{x\in \mathbb{R}^{\ast
}:a\leq x\leq b\};$

\item $\left( a,b\right) _{\mathbb{R}^{\ast }}=\{x\in \mathbb{R}^{\ast
}:a<x<b\};$

\item $]a,b[\ =\left[ a,b\right] _{\mathbb{R}^{\ast }}\setminus \left( 
\mathfrak{mon}(a)\cup \mathfrak{mon}(b)\right) .$
\end{itemize}
\end{itemize}

\section{Definition of Ultrafunctions}

In this section we introduce a few Desideratum that will be used to
introduce ultrafunctions in a slightly different way with respect to what we
did in \cite{belu2012}, \cite{belu2013}.

Let $\mathfrak{X=}\mathcal{P}_{fin}(\mathfrak{F}\left( \mathbb{R},\mathbb{R}%
\right) ).$ Given $\lambda \in \mathfrak{X,}$ we set $V_{\lambda
}=\{Span\left( f_{j}\right) \ |\ f_{j}\in \lambda \}.$

\begin{definition}
An internal function 
\begin{equation*}
u=\lim_{\lambda \uparrow \Lambda }u_{\lambda }\in \mathfrak{F}\left( \mathbb{%
R}\right) ^{\ast }
\end{equation*}%
is called ultrafunction if, for every $\lambda \in \mathfrak{X,}$ $%
u_{\lambda }\in V_{\lambda }.$ The space of ultrafunctions will be denoted
by $\widetilde{\mathfrak{F}\left( \mathbb{R}\right) }$. With some abuse of
notation we will call ultrafunction also the restriction of $u$ to any
internal subset of $\mathbb{R}^{\ast }$.
\end{definition}

In particular, we have that 
\begin{equation*}
\widetilde{\mathfrak{F}\left( \mathbb{R}\right) }=\lim_{\lambda \uparrow
\Lambda }V_{\lambda },
\end{equation*}%
so, being a $\Lambda $-limit of finite dimensional vector spaces, the vector
space of ultrafunctions has hyperfinite dimension. Moreover, given any
vector space of functions $W\subset \mathfrak{F}\left( \mathbb{R}\right) $,
we can define the space of ultrafunctions generated by $W$ as follows:%
\begin{equation*}
\widetilde{W}=W^{\ast }\cap \widetilde{\mathfrak{F}\left( \mathbb{R}\right) }%
.
\end{equation*}

Let us observe that 
\begin{equation*}
\widetilde{W}=\lim_{\lambda \uparrow \Lambda }W_{\lambda },
\end{equation*}
where for every $\lambda \in \mathfrak{X}$ we pose $W_{\lambda }=V_{\lambda
}\cap W.$

The space of ultrafunctions $\widetilde{\mathfrak{F}\left( \mathbb{R}\right) 
}$ is too large for applications. We want to have a smaller space $\mathfrak{%
U}(\mathbb{R})\subset \widetilde{\mathfrak{F}\left( \mathbb{R}\right) }$
which satisfies suitable properties for applications. We list the main
properties that we would like to obtain for $\mathfrak{U}(\mathbb{R}).$

\begin{desideratum}
\label{a}There is an infinite number $\beta $ such that if $u(x)\in 
\mathfrak{U}(\mathbb{R}),$ then $u(x)=0$ for $|x|>\beta $ and $u(x)\in
L^{\infty }(\mathbb{R})^{\ast }.$
\end{desideratum}

Desideratum \ref{a} states that the ultrafunctions have an uniform compact
support and are bounded in $\mathbb{R}^{\ast }$. From these conditions it
follows that, if $u(x)\in \mathfrak{U}(\mathbb{R}),$ then $u(x)\in L^{p}(%
\mathbb{R})^{\ast }$ for every $p;$ in particular, $u(x)$ is summable and it
is in $L^{2}(\mathbb{R})^{\ast }.$ So $\mathfrak{U}(\mathbb{R})\subset L^{2}(%
\mathbb{R})^{\ast },$ and this allows to give to $\mathfrak{U}(\mathbb{R})$
the euclidean structure and the norm induced by $L^{2}(\mathbb{R})^{\ast }.$

\begin{desideratum}
\label{b}$\mathfrak{U}(\mathbb{R})\subset F_{\sharp }\left( \mathbb{R}%
\right) ^{\ast },$ where%
\begin{equation*}
F_{\sharp }\left( \mathbb{R}\right) =\left\{ u\in L_{loc}^{1}\ \mid \ u(x)=\ 
\underset{\varepsilon \rightarrow 0^{+}}{\lim }\frac{1}{2\varepsilon }%
\int_{x-\varepsilon }^{x+\varepsilon }u(y)\ dy\right\} .
\end{equation*}
\end{desideratum}

This request, which may seem strange at first sight, will allow to associate
to every point $a\in \lbrack -\beta ,\beta ]$ a delta (or Dirac)
ultrafunction centered in $a$, namely an ultrafunction $\delta _{a}(x)$ such
that, for every ultrafunction $u(x)$, we have 
\begin{equation*}
\int^{\ast }u(x)\delta _{a}(x)dx=u(a).
\end{equation*}

\begin{desideratum}
\label{c}If $f\in \mathcal{C}^{1}\left( \mathbb{R}\right) ,\ $and $a,b\in 
\mathbb{R}$, then $\left( f\cdot \chi _{\left[ a,b\right] }\right) ^{\ast
}\in \mathfrak{U}(\mathbb{R}).$
\end{desideratum}

Desideratum \ref{c} is introduced for a few different reasons. First of all,
it is important to have the characteristic functions of intervals even if,
due to Desideratum \ref{b}, we will have to pay attenction in choosing the
right definition of characteristic functions; moreover, it is important to
have the extensions of $\mathcal{C}^{1}$ functions in $\mathfrak{U}(\mathbb{R%
})$ (one could ask this property for continuous function but, as we will
show later, this request seems difficult to obtain if we want also the other
Desideratums that we are presenting here). Finally, we will show that from
Desideratum \ref{c} it follows that the delta functions have compact support
concentrated around their center: in fact we will show that, $\forall a\in 
\mathfrak{gal}(0),$ $\mathfrak{supp}\left( \delta _{a}\right) \subset 
\mathfrak{mon}(a).$

However it would be nice to have the previous property in the following more
general fashion:

\begin{desideratum}
\label{c++}$\forall a\in \left[ -\beta ,\beta \right] ,$ $\mathfrak{supp}%
\left( \delta _{a}\right) \subset \mathfrak{mon}(a).$
\end{desideratum}

Our next desideratum is the following:

\begin{desideratum}
\label{e}There exists a linear map $\widetilde{\left( \cdot \right) }:\left[
L_{loc}^{1}\left( \Omega \right) \right] ^{\ast }\rightarrow \mathfrak{U}(%
\mathbb{R})$ such that $\forall f\in \left[ L_{loc}^{1}\left( \Omega \right) %
\right] ^{\ast },\ \forall v\in \mathfrak{U}(\mathbb{R}),$ we have%
\begin{equation*}
\int^{\ast }fvdx=\int^{\ast }\widetilde{f}vdx.
\end{equation*}
\end{desideratum}

Desideratum \ref{e} substantially states that it is possible to define the
projection of an $\left[ L_{loc}^{1}\left( \Omega \right) \right] ^{\ast }$
function on $\mathfrak{U}(\mathbb{R}).$ In particular, this is useful to
associate canonically an ultrafunction to every function $f\in
L_{loc}^{1}\left( \Omega \right) $ since, in general, it will be false that $%
f^{\ast }\in \mathfrak{U}(\mathbb{R})$ (but when $f^{\ast }\in \mathfrak{U}(%
\mathbb{R})$ by Desideratum \ref{e} we have $f^{\ast }=\widetilde{f}$).

\begin{desideratum}
\label{f}There exists a map $D:\mathfrak{U}(\mathbb{R})\rightarrow \mathfrak{%
U}(\mathbb{R})$ such that

\begin{itemize}
\item $\forall f\in \mathcal{C}^{1}\left( \mathbb{R}\right) ,\ \forall x\in 
\mathbb{R},\ D\widetilde{f}(x)=\widetilde{f^{\prime }}(x);$

\item $\forall u,v\in \mathfrak{U}(\mathbb{R}),\ \int_{-\beta }^{\beta
}Du(x)v(x)dx=-\int_{-\beta }^{\beta }u(x)Dv(x)dx+\left[ u(x)v(x)\right]
_{-\beta }^{\beta };$

\item $D\widetilde{1}=0;$

\item $D\chi _{\left[ a,b\right] }=\delta _{a}-\delta _{b}.$
\end{itemize}
\end{desideratum}

Desideratum \ref{f} simply states that it is possible to define a derivative
on $\mathfrak{U}(\mathbb{R})$ which satisfies a few expected properties.

In the next sections we show how to construct a space that satisfies all the
Desideratum that we presented.

\section{Construction of a canonical space of ultrafunctions}

We want to consider a special subset of ultrafunctions. Let $\beta $ be an
infinite number; we set%
\begin{equation*}
\Gamma =\left\{ \gamma _{0},\gamma _{1},...,\gamma _{\ell }\right\} \subset 
\mathbb{R}^{\ast },
\end{equation*}%
where $\gamma _{0}=-\beta ;\ \gamma _{\ell }=\beta $ and,$\ $for $%
j=0,1,...,\ell -1,$ we require that%
\begin{equation*}
0<\gamma _{j+1}-\gamma _{j}<\eta
\end{equation*}%
where $\eta $ is an infinitesimal number. Moreover, it is useful to assume
that $\mathbb{R}\subseteq \Gamma .$

For $j=0,1,...,\ell -1,$ we set 
\begin{equation*}
\mathbb{I}_{j}:=\left( \gamma _{j},\gamma _{j+1}\right) _{\mathbb{R}^{\ast
}}.
\end{equation*}

For every $a,b\in \Gamma $ we denote by $\chi _{\left[ a,b\right] }(x)$ the
characteristic function of $\left[ a,b\right] $ defined in a slightly
different way:

\begin{equation}
\chi _{\left[ a,b\right] }(x)=\left\{ 
\begin{array}{cc}
1 & \text{if\ }\ x\in \left( a,b\right) \\ 
0 & \text{if\ }\ x\notin \left[ a,b\right] \\ 
\frac{1}{2} & \text{if\ }\ x=a,b;\ a\neq -\beta ;\ b\neq \beta \\ 
1 & \text{if\ }\ x=a=-\beta \\ 
1 & \text{if\ }\ x=b=\beta%
\end{array}%
\right. ;  \label{chi}
\end{equation}%
For every $j=0,1,...,\ell -1,$ we set%
\begin{equation*}
\chi _{j}(x)=\chi _{_{\mathbb{I}_{j}}}(x).
\end{equation*}

The set of functions%
\begin{equation*}
\mathfrak{G}=\left\{ \dsum\limits_{j=0}^{\ell -1}c_{j}\chi _{j}(x)\ |\
c_{j}\in \mathbb{R}^{\ast }\right\}
\end{equation*}%
will be referred to as the set of \textbf{grid functions.}

\begin{definition}
\label{cuf}We denote by $\mathfrak{U}(\mathbb{R})$ the space of
ultrafunctions 
\begin{equation*}
u:\left[ -\beta ,\beta \right] \rightarrow \mathbb{R}^{\ast }
\end{equation*}%
which can be represented as follows:%
\begin{equation*}
u(x)=\dsum\limits_{j=0}^{\ell -1}v_{j}(x)\chi _{j}(x)
\end{equation*}%
where, $\forall j\in J,\ v_{j}(x)\in \widetilde{\mathcal{C}^{1}(\mathbb{R})}$%
. We will refer to $\mathfrak{U}(\mathbb{R})$ as the canonical space of
ultrafunctions.
\end{definition}

\begin{prop}
The elements of $\mathfrak{U}(\mathbb{R})$ are restriction to $\left[ -\beta
,\beta \right] $ of ultrafunctions.
\end{prop}

\begin{proof} Let $u(x)=\dsum\limits_{j=0}^{\ell }v_{j}(x)\chi _{j}(x),$
let $\ell =\lim_{\lambda \uparrow \Lambda }\ell _{\lambda },$ $\chi
_{j}(x)=\lim_{\lambda \uparrow \Lambda }\chi _{j,\lambda }(x)$ and $%
v_{j}(x)=\lim_{\lambda \uparrow \Lambda }v_{j,\lambda }.$ Then 
\begin{equation*}
u(x)=\lim_{\lambda \uparrow \Lambda }\dsum\limits_{j=0}^{\ell _{\lambda
}-1}v_{j,_{\lambda }}(x)\chi _{j,\lambda }(x),
\end{equation*}%
so it is an ultrafunction. \end{proof}

\begin{prop}
$\mathfrak{U}(\mathbb{R})$ is an hyperfinite dimensional vector space, and $%
\dim (\mathfrak{U}(\mathbb{R}))\leq \ell \cdot \dim \widetilde{\mathcal{C}%
^{1}(\mathbb{R})}.$
\end{prop}

\begin{proof} If $B=\{v_{i}(x)\mid i\leq \dim (\widetilde{\mathcal{C}^{1}(%
\mathbb{R})})\}$ is a basis for $\widetilde{\mathcal{C}^{1}(\mathbb{R})}$,
the set 
\begin{equation*}
B_{V}=\{v_{i}(x)\chi _{j}(x)\mid v_{i}\in B,\text{ }j=0,...,\ell \}
\end{equation*}%
is a set of generators for $\mathfrak{U}(\mathbb{R}),$ and its cardinality
is $\ell \cdot \dim \widetilde{\mathcal{C}^{1}(\mathbb{R})}.$ So $\dim (%
\mathfrak{U}(\mathbb{R}))\leq \ell \cdot \dim \widetilde{\mathcal{C}^{1}(%
\mathbb{R})}.$ \end{proof}

Since $\mathfrak{U}(\mathbb{R})\subset \left[ L^{2}(\mathbb{R})\right]
^{\ast },$ it can be equipped with the following scalar product%
\begin{equation*}
\left( u,v\right) =\int^{\ast }u(x)\overline{v(x)}\ dx,
\end{equation*}%
where $\int^{\ast }$ is the natural extension of the Lebesgue integral
considered as a functional%
\begin{equation*}
\int :L^{1}(\Omega )\rightarrow {\mathbb{C}}.
\end{equation*}%
The norm of a (canonical) ultrafunction will be given by 
\begin{equation*}
\left\Vert u\right\Vert =\left( \int^{\ast }|u(x)|^{2}\ dx\right) ^{\frac{1}{%
2}}.
\end{equation*}

Canonical ultrafunctions have a few interesting properties:

\begin{prop}
\label{nina}The following properties hold:

\begin{enumerate}
\item If $f\in \mathcal{C}^{1}(\mathbb{R})$ then $f^{\ast }\cdot \chi _{%
\left[ -\beta ,\beta \right] _{\mathbb{R}^{\ast }}}\in \mathfrak{U}(\mathbb{R%
});$

\item if $u\in \mathfrak{U}(\mathbb{R})$ and $a,b\in \Gamma $, then $u\cdot
\chi _{\left[ a,b\right] _{\mathbb{R}^{\ast }}}\in \mathfrak{U}(\mathbb{R});$

\item if $u\in \mathfrak{U}(\mathbb{R})$ then for $j=1,...,\ell -1$ the
limits 
\begin{equation*}
\left( \underset{x\rightarrow \gamma _{j}^{\pm }}{\lim }\right) ^{\ast }u(x)
\end{equation*}%
are well defined and we set%
\begin{equation}
u(\gamma _{j}^{+}):=\ \left( \underset{x\rightarrow \gamma _{j}^{+}}{\lim }%
\right) ^{\ast }u(x);\ u(\gamma _{j}^{-}):=\ \left( \underset{x\rightarrow
\gamma _{j}^{-}}{\lim }\right) ^{\ast }u(x);  \label{c+}
\end{equation}

\item if $u\in \mathfrak{U}(\mathbb{R})$ then for $j=0$ the limit 
\begin{equation*}
\left( \underset{x\rightarrow \gamma _{0}^{+}}{\lim }\right) ^{\ast }u(x)
\end{equation*}%
is well defined and for $j=l$ the limit 
\begin{equation*}
\left( \underset{x\rightarrow \gamma _{l}^{-}}{\lim }\right) ^{\ast }u(x)
\end{equation*}%
is well defined.

\item if, for every $j=0,...,\ell $ we set 
\begin{equation*}
V(\mathbb{I}_{j}):=\{u(x)\chi _{j}(x)\mid u(x)\in \widetilde{\mathcal{C}^{1}(%
\mathbb{R})}\},
\end{equation*}%
then, for $k\neq j,$ $V(\mathbb{I}_{j})$ and $V(\mathbb{I}_{k})$ are
orthogonal;

\item $\mathfrak{U}(\mathbb{R})$ can be splitted in orthogonal spaces as
follows:%
\begin{equation*}
\mathfrak{U}(\mathbb{R})=\dbigoplus\limits_{j=0}^{\ell }V(\mathbb{I}_{j}).
\end{equation*}
\end{enumerate}
\end{prop}

\begin{proof} 1)\textbf{\ }If $f\in \mathcal{C}^{1}(\mathbb{R}),$ then $%
f^{\ast }\in \widetilde{\mathcal{C}^{1}(\mathbb{R})},$ and 
\begin{equation*}
f^{\ast }\cdot \chi _{\left[ -\beta ,\beta \right] _{\mathbb{R}^{\ast
}}}=\dsum\limits_{j=0}^{\ell -1}f^{\ast }(x)\chi _{j}(x)\in \mathfrak{U}(%
\mathbb{R}).
\end{equation*}

2) It follows by (1).

3) If $u(x)=\dsum\limits_{j=0}^{\ell -1}u_{j}(x)\chi _{j}(x),$ then 
\begin{equation*}
u(\gamma _{0}^{-})=\left( \underset{x\rightarrow \gamma _{j}^{-}}{\lim }%
\right) ^{\ast }u_{j-1}(x)
\end{equation*}%
and 
\begin{equation*}
u(\gamma _{j}^{+})=\left( \underset{x\rightarrow \gamma _{j}^{+}}{\lim }%
\right) ^{\ast }u_{j}(x)
\end{equation*}%
and these limits exist because $u_{j-1},u_{j\text{ }}$ are continuous on $%
\overline{\mathbb{I}_{j-1}},\overline{\mathbb{I}_{j}}$ respectively.

4) The same as in 1).

5) This is immediate since, if $j\neq k$, if $u\in V(\mathbb{I}_{j})$ and $%
v\in V(\mathbb{I}_{k})$ then the supports of $u$ and $v$ are disjoint.

6) Having proved 3), it remains only to prove that $\dbigoplus%
\limits_{j=0}^{\ell }V(\mathbb{I}_{j})$ generates all $\mathfrak{U}(\mathbb{R%
})$, and this is clear because, if $u(x)=\dsum\limits_{j=0}^{\ell
}u_{j}(x)\chi _{j}(x)$ then, for every $j=0,...,\ell -1$, $u_{j}(x)\chi
_{j}(x)\in V(\mathbb{I}_{j}).$ \end{proof}

\begin{definition}
\label{split}A basis $\left\{ e_{j,k}:j=0,...,\ell -1,\
k=1,...,s_{j}\right\} $ for $\mathfrak{U}(\mathbb{R})$ is called \textbf{%
splitted} \textbf{basis} if, for every $j=0,...,\ell ,$ $\left\{
e_{j,k}\right\} _{k=1}^{s_{j}}$ is a basis for $V(\mathbb{I}_{j}).$
\end{definition}

\section{Delta and Sigma basis}

Following the approach presented in \cite{belu2013}, in this section we
introduce two particular bases for $\mathfrak{U}(\mathbb{R})$ and we study
their main properties. We start by defining the \textit{Delta ultrafunctions}%
. In order to do this, it is useful to observe that the value of an
ultrafunction $u$ for $\gamma _{j},$ $j=1,...,\ell -1,$ can be defined as
follows: 
\begin{equation*}
u(\gamma _{j})=\frac{u(\gamma _{j}^{+})+u(\gamma _{j}^{-})}{2}
\end{equation*}%
where $u(x^{+}),u(x^{-})$ are defined by (\ref{c+}). The fact that this
definition makes sense follows by points 3) and 4) in Proposition \ref{nina}%
. Moreover we pose 
\begin{equation*}
u(\gamma _{0})=u(\mathbb{-}\beta )=u^{+}(\mathbb{-}\beta );\ u(\gamma _{\ell
})=u(\beta )=u^{-}(\beta ).
\end{equation*}%
These observations are relevant in the following definition:

\begin{definition}
\label{dede}Given a number $q\in \mathbb{[-}\beta ,\beta ]$ we denote by $%
\delta _{q}(x)$ an ultrafunction in $\mathfrak{U}(\mathbb{R})$ such that 
\begin{equation}
\forall v\in \mathfrak{U}(\mathbb{R}),\ \int^{\ast }v(x)\delta
_{q}(x)dx=v(q).  \label{deltafunction}
\end{equation}%
$\delta _{q}(x)$ is called the Delta (or Dirac) ultrafunction concentrated
in $q$.
\end{definition}

Let us see the main properties of the Delta ultrafunctions:

\begin{theorem}
\label{delta} We have the following properties:

\begin{enumerate}
\item For every $q\in \mathbb{[-}\beta ,\beta ]$ there exists an unique
Delta ultrafunction concentrated in $q;$

\item for every $a,\ b\in \mathbb{[-}\beta ,\beta ]\ \delta _{a}(b)=\delta
_{b}(a);$

\item $\left\Vert \delta _{q}\right\Vert ^{2}=\delta _{q}(q).$
\end{enumerate}
\end{theorem}

\begin{proof} 1) Let $\left\{ e_{j,k}:j=0,...,\ell -1,\
k=1,...,s_{j}\right\} $ be an orthogonal splitted basis of $\mathfrak{U}(%
\mathbb{R})$ (see Def. \ref{split}). If $q\in \mathbb{I}_{j}$ we pose%
\begin{equation*}
\delta _{q}(x)=\sum_{k=1}^{s_{j}}e_{j,k}(q)e_{j,k}(x).
\end{equation*}

For every $i\neq j$, for every $v\in V\left( \mathbb{I}_{i}\right) $ we have 
$\int^{\ast }v(x)\delta _{q}(x)dx=0=v(q).$ If $v\in V\left( \mathbb{I}%
_{j}\right) ,$ $v=\sum_{k=1}^{s_{j}}v_{k}e_{j,k}(x)$ we have 
\begin{eqnarray*}
\int^{\ast }v(x)\delta _{q}(x)dx &=& \\
\int^{\ast }\left( \sum_{k=1}^{s_{j}}e_{j,k}(q)e_{j,k}(x)\right) \left(
\sum_{k=1}^{s_{j}}v_{k}e_{j,k}(x)\right) dx &=&\sum_{k=1}^{s_{j}}\int^{\ast
}e_{j,k}(q)e_{j,k}(x)v_{k}e_{j,k}(x)= \\
\sum_{k=1}^{s_{j}}e_{j,k}(q)v_{k} &=&v(q).
\end{eqnarray*}

If $q=\gamma _{0}$ we pose 
\begin{equation*}
\delta _{q}(x)=\sum_{k=1}^{s_{0}}e_{j,k}^{+}(q)e_{j,k}(x)
\end{equation*}

and if $q=\gamma _{\ell \text{ }}$we pose 
\begin{equation*}
\delta _{q}(x)=\sum_{k=1}^{s_{\ell -1}}e_{j,k}^{-}(q)e_{j,k}(x).
\end{equation*}

The verification that these definitions are well posed is equal to the one
carried out for $q\in \mathbb{I}_{j}.$

If $q=\gamma _{j},$ $j\neq 0,\ell $ we set%
\begin{equation*}
\delta _{q}(x)=\frac{1}{2}\left(
\sum_{k=1}^{s_{j-1}}e_{j-1,k}^{-}(q)e_{j-1,k}(x)+%
\sum_{k=1}^{s_{j}}e_{j,k}^{+}(q)e_{j,k}(x)\right) .
\end{equation*}%
Then%
\begin{equation*}
\int^{\ast }v(x)\delta _{q}(x)dx=
\end{equation*}%
\begin{eqnarray*}
\frac{1}{2}\left( \int_{\left[ \gamma _{j-1},\gamma _{j}\right] }^{\ast
}v(x)\left( \sum_{k=1}^{s_{j}}e_{j-1,k}^{-}(q)e_{j-1,k}(x)\right) dx+\int_{%
\left[ \gamma _{j},\gamma _{j+1}\right] }^{\ast }v(x)\left(
\sum_{k=1}^{s_{j}}e_{j,k}^{+}(q)e_{j,k}(x)\right) dx\right) &=& \\
\frac{1}{2}\left( \int_{\left[ \gamma _{j-1},\gamma _{j}\right] }^{\ast
}v^{-}(x)\left( \sum_{k=1}^{s_{j}}e_{j-1,k}^{-}(q)e_{j-1,k}(x)\right)
dx+\int_{\left[ \gamma _{j},\gamma _{j+1}\right] }^{\ast }v^{+}(x)\left(
\sum_{k=1}^{s_{j}}e_{j,k}^{+}(q)e_{j,k}(x)\right) dx\right) &=&
\end{eqnarray*}%
\begin{equation*}
\frac{1}{2}\left[ v^{-}(\gamma _{j})+v^{+}(\gamma _{j})\right] =v(\gamma
_{j}).
\end{equation*}

The Delta function in $q$ is unique: if $f_{q}(x)$ is another Delta
ultrafunction centered in $q$ then for every $y\in \mathbb{[-}\beta ,\beta ]$
we have:%
\begin{equation*}
\delta _{q}(y)-f_{q}(y)=\int^{\ast }(\delta _{q}(x)-f_{q}(x))\delta
_{y}(x)dx=\delta _{y}(q)-\delta _{y}(q)=0
\end{equation*}

and hence $\delta _{q}(y)=f_{q}(y)$ for every $y\in \mathbb{(-}\beta ,\beta
).$

2.$\ \delta _{a}\left( b\right) =\int^{\ast }\delta _{a}(x)\delta _{b}(x)\
dx=\delta _{b}\left( a\right) .$

3. $\left\Vert \delta _{q}\right\Vert ^{2}=\int^{\ast }\delta _{q}(x)\delta
_{q}(x)=\delta _{q}(q)$. \end{proof}

Let us observe that, as the previous proof shows, in every point $\gamma _{j%
\text{ }}$ of the grid $\Gamma ,$ with the exceptions of $-\beta ,\beta ,$
it is possible to define three delta functions centered in $\gamma _{j}$,
namely $\delta _{\gamma _{j}}^{-}(x),\delta _{\gamma _{j}}^{+}(x)$ and $%
\delta _{\gamma _{j}}(x),$ which satisfy the following properties: for every 
$\forall v\in \mathfrak{U}(\mathbb{R}),\ $we have%
\begin{eqnarray}
\int^{\ast }v(x)\delta _{\gamma _{j}}^{-}(x)dx &=&v^{-}(\gamma _{j});  \notag
\\
\int^{\ast }v(x)\delta _{\gamma _{j}}^{+}(x)dx &=&v^{+}(\gamma _{j});
\label{yole} \\
\int^{\ast }v(x)\delta _{\gamma _{j}}(x)dx &=&v(\gamma _{j}).  \notag
\end{eqnarray}

Moreover, it is immediate to prove that the conditions in $\left( \ref{yole}%
\right) $ charatecterize uniquely the functions $\delta _{\gamma
_{j}}^{-}(x),$ $\delta _{\gamma _{j}}^{+}(x)$ and $\delta _{\gamma _{j}}(x).$
So we will consider $\left( \ref{yole}\right) $ as a definition for $\delta
_{\gamma _{j}}^{-}(x),$ $\delta _{\gamma _{j}}^{+}(x)$ and $\delta _{\gamma
_{j}}(x).$

\begin{definition}
A Delta-basis $\left\{ \delta _{a}(x)\right\} _{a\in \Sigma }$ $(\Sigma
\subset \mathbb{[-}\beta ,\beta ])$ is a basis for $\mathfrak{U}(\mathbb{R})$
whose elements are Delta ultrafunctions. Its dual basis $\left\{ \sigma
_{a}(x)\right\} _{a\in \Sigma }$ is called Sigma-basis. We recall that, by
definition of dual basis, for every $a,b\in \Sigma $ the equation%
\begin{equation}
\int^{\ast }\delta _{a}(x)\sigma _{b}(x)dx=\delta _{ab}  \label{mimma}
\end{equation}%
holds. A set $A\subset \mathbb{[-}\beta ,\beta ]$ is called set of
independent points if $\left\{ \delta _{a}(x)\right\} _{a\in A}$ is a basis
\end{definition}

The existence of a Delta-basis is an immediate consequence of the following
fact:

\begin{remark}
The set $\left\{ \delta _{a}(x)|a\in \lbrack \mathbb{-}\beta ,\beta
]\right\} $ generates all $\mathfrak{U}(\mathbb{R}).$ In fact, let $G(\Omega
)$ be the vector space generated by the set $\left\{ \delta _{a}(x)\ |\ a\in
\lbrack \mathbb{-}\beta ,\beta ]\right\} $ and suppose that $G(\Omega )$ is
properly included in $\mathfrak{U}(\mathbb{R}).$ Then the orthogonal $%
G(\Omega )^{\perp }$ of $G(\Omega )$ in $\mathfrak{U}(\mathbb{R})$ contains
a function $f\neq 0.$ But, since $f\in $ $G(\Omega )^{\perp },$ for every $%
a\in \mathbb{[-}\beta ,\beta ]$ we have 
\begin{equation*}
f(a)=\int^{\ast }f(x)\delta _{a}(x)dx=0,
\end{equation*}%
so $f_{\upharpoonleft _{\lbrack \mathbb{-}\beta ,\beta ]}}=0$ and this is
absurd. Thus the set $\left\{ \delta _{a}(x)\ |\ a\in \mathbb{(-}\beta
,\beta )\right\} $ generates $\mathfrak{U}(\mathbb{R}),$ hence it contains a
basis.
\end{remark}

Let us see some properties of Delta- and Sigma-bases (which, in this new
context, are slightly different from the one presented in \cite{belu2013}):

\begin{theorem}
\label{tbase}A Delta-basis $\left\{ \delta _{q}(x)\right\} _{q\in \Sigma }$
and its dual basis $\left\{ \sigma _{q}(x)\right\} _{q\in \Sigma }$ satisfy
the following properties:

\begin{enumerate}
\item if $u\in \mathfrak{U}(\mathbb{R})$ then%
\begin{equation*}
u(x)=\sum_{q\in \Sigma }\left( \int^{\ast }\sigma _{q}(\xi )u(\xi )d\xi
\right) \delta _{q}(x);
\end{equation*}

\item if $u\in \mathfrak{U}(\mathbb{R})$ then%
\begin{equation}
u(x)=\sum_{q\in \Sigma }u(q)\sigma _{q}(x);  \label{brava+}
\end{equation}

\item if two ultrafunctions $u$ and $v$ coincide on a set of independent
points then they are equal;

\item if $\Sigma $ is a set of independent points and $a,b\in \Sigma $ then $%
\sigma _{a}(b)=\delta _{ab};$

\item for every $q\in \lbrack -\beta ,\beta ],$ $\sigma _{q}(x)$ is well
defined;

\item for every $q\in \mathbb{[-}\beta ,\beta ]$\ if $q\in \mathbb{I}_{j%
\text{ }}$ then $\mathfrak{supp}(\delta _{q}(x))\subset \overline{\mathbb{I}%
_{j}}$ and $\mathfrak{supp}(\sigma _{q}(x))\subset \overline{\mathbb{I}_{j}}%
; $

\item for every $\gamma _{j}\in \Gamma \setminus \{\gamma _{0},\gamma _{\ell
}\},$\ $\mathfrak{supp}(\delta _{\gamma _{j}}(x))\subset \overline{\mathbb{I}%
_{j-1}\cup \mathbb{I}_{j}}$ and $\mathfrak{supp}(\sigma _{\gamma
_{j}}(x))\subset \overline{\mathbb{I}_{j-1}\cup \mathbb{I}_{j}};$

\item $\mathfrak{supp}(\delta _{\gamma _{0}}(x))\subset \overline{\mathbb{I}%
_{0}}$, $\mathfrak{supp}(\sigma _{\gamma _{0}}(x))\subset \overline{\mathbb{I%
}_{0}},$ $\mathfrak{supp}(\delta _{\gamma _{\ell }}(x))\subset \overline{%
\mathbb{I}_{\ell }}$ and $\mathfrak{supp}(\sigma _{\gamma _{\ell
}}(x))\subset \overline{\mathbb{I}_{\ell -1}};$

\item for every $q\in \left[ -\beta ,\beta \right] ,$\ $\mathfrak{supp}%
(\delta _{q}(x))\subset \mathfrak{mon}\left( q\right) $ and $\mathfrak{supp}%
(\sigma _{q}(x))\subset \mathfrak{mon}\left( q\right) .$
\end{enumerate}
\end{theorem}

\begin{proof} 1) It is an immediate consequence of the definition of dual
basis.

2) Since $\left\{ \delta _{q}(x)\right\} _{q\in \Sigma }$ is the dual basis
of $\left\{ \sigma _{q}(x)\right\} _{q\in \Sigma }$ we have that%
\begin{equation*}
u(x)=\sum_{q\in \Sigma }\left( \int \delta _{q}(\xi )u(\xi )d\xi \right)
\sigma _{q}(x)=\sum_{q\in \Sigma }u(q)\sigma _{q}(x).
\end{equation*}

3) It follows directly from the previous point.

4) If follows directly by equation (\ref{mimma}).

5) Given any point $q\in (-\beta ,\beta )$ clearly there is a Delta-basis $%
\left\{ \delta _{a}(x)\right\} _{a\in \Sigma }$ with $q\in \Sigma .$ Then $%
\sigma _{q}(x)$ can be defined by mean of the basis $\left\{ \delta
_{a}(x)\right\} _{a\in \Sigma }.$ We have to prove that, given another Delta
basis $\left\{ \delta _{a}(x)\right\} _{a\in \Sigma ^{\prime }}$ with $q\in
\Sigma ^{\prime },$ the corresponding $\sigma _{q}^{\prime }(x)$ is equal to 
$\sigma _{q}(x).$ Using (2), with $u(x)=\sigma _{q}^{\prime }(x),$ we have
that%
\begin{equation*}
\sigma _{q}^{\prime }(x)=\sum_{a\in \Sigma }\sigma _{q}^{\prime }(a)\sigma
_{a}(x).
\end{equation*}%
Then, by (4), it follows that $\sigma _{q}^{\prime }(x)=\sigma _{q}(x).$

6) As we proved in Theorem \ref{delta}, \ if $q\in \mathbb{I}_{j}$ then $%
\delta _{q}$ is an element of $V(\mathbb{I}_{j}),$ so $\mathfrak{supp}%
(\delta _{q}(x))\subset \overline{\mathbb{I}_{j}}.$ Now $\delta _{q}\in V(%
\mathbb{I}_{j}),$ so there is a corrispective function $\sigma _{q}\in V(%
\mathbb{I}_{j})$ which is the sigma function centered in $q.$ If we extend
this function to $[-\beta ,\beta ]$ by posing $\sigma _{q}(x)=0$ for $%
x\notin \mathbb{I}_{j}$ we obtain, by uniqueness, exactly the sigma function
centered in $q$ in $\mathfrak{U}(\mathbb{R})$. And $\mathfrak{supp}(\sigma
_{q}(x))\subset \overline{\mathbb{I}_{j}}.$

7) In Theorem \ref{delta} we proved that $\delta _{\gamma _{j}}$ is an
element in $V(\mathbb{I}_{j})\cup V(\mathbb{I}_{j+1}),$ so $\mathfrak{supp}%
(\delta _{\gamma _{j}}(x))\subset \overline{\mathbb{I}_{j-1}\cup \mathbb{I}%
_{j}}.$ Now we can consider its corrispective sigma function $\sigma
_{\gamma _{j}}\in V(\mathbb{I}_{j})\cup V(\mathbb{I}_{j+1}).$ If we extend
this function to $\mathfrak{U}(\mathbb{R})$ by posing $\sigma _{\gamma
_{j}}(x)=0$ for $x\notin \mathbb{I}_{j}\cup \mathbb{I}_{j+1}$, we obtain the
sigma function in $\mathfrak{U}(\mathbb{R})$ centered in $\gamma _{j}.$ And,
by construction, $\mathfrak{supp}(\sigma _{\gamma _{j}}(x))\subset \overline{%
\mathbb{I}_{j}\cup \mathbb{I}_{j+1}}.$

8) In Theorem \ref{delta} we proved that $\delta _{0}$ is an element in $V(%
\mathbb{I}_{0})$ and $\delta _{\ell }$ is in $V(\mathbb{I}_{\ell -1}),$ and
that the same property holds for the corrispective $\sigma $ functions can
be proved as in point (6) of this Theorem. So $\mathfrak{supp}(\delta
_{\gamma _{0}}(x))\subset \overline{\mathbb{I}_{0}}$, $\mathfrak{supp}%
(\sigma _{\gamma _{0}}(x))\subset \overline{\mathbb{I}_{0}},$ $\mathfrak{supp%
}(\delta _{\gamma _{\ell +1}}(x))\subset \overline{\mathbb{I}_{\ell }}$ and $%
\mathfrak{supp}(\sigma _{\gamma _{\ell }}(x))\subset \overline{\mathbb{I}}%
_{\ell -1}.$

9) It is a straightforward consequence of the points 6 and 7, since for
every $j\in J$ we have $\mathbb{I}_{j}\cup \mathbb{I}_{j+1}\subset \mathfrak{%
mon}\left( q\right).$ \end{proof}

\section{Canonical extension of functions}

We start by defining a map%
\begin{equation*}
\widetilde{\left( \cdot \right) }:\left[ L_{loc}^{1}\left( \mathbb{R}\right) %
\right] ^{\ast }\rightarrow \mathfrak{U}(\mathbb{R})
\end{equation*}%
which will be very useful in the extension of functions.

\begin{definition}
\label{tilde} If $u\in \left[ L_{loc}^{1}\left( \mathbb{R}\right) \right]
^{\ast },$ $\widetilde{u}$ denotes the unique ultrafunction such that 
\begin{equation*}
\forall v\in \mathfrak{U}(\mathbb{R}),\ \int^{\ast }\widetilde{u}%
(x)v(x)dx=\int^{\ast }u(x)v(x)dx.
\end{equation*}
\end{definition}

\begin{remark}
\label{rina} Notice that, if $u\in \left[ L^{2}\left( \mathbb{R}\right) %
\right] ^{\ast },$ then $\widetilde{u}=P_{V}u$ where%
\begin{equation*}
P_{V}:\left[ L^{2}\left( \mathbb{R}\right) \right] ^{\ast }\rightarrow 
\mathfrak{U}(\mathbb{R})
\end{equation*}%
is the orthogonal projection.
\end{remark}

The following theorem shows that $\widetilde{u}$ is well defined and unique.

\begin{theorem}
\label{CA} If $u\in \left[ L_{loc}^{1}\left( \mathbb{R}\right) \right]
^{\ast }$ then 
\begin{eqnarray}
\widetilde{u}(x) &=&\sum_{q\in \Sigma }\left[ \int u(\xi )\delta _{q}(\xi
)d\xi \right] \sigma _{q}(x)  \label{bella} \\
&=&\sum_{q\in \Sigma }\left[ \int u(\xi )\sigma _{q}(\xi )d\xi \right]
\delta _{q}(x).  \label{bella+}
\end{eqnarray}
\end{theorem}

\begin{proof} It is sufficient to prove that 
\begin{equation*}
\forall v\in \mathfrak{U}(\mathbb{R}),\ \int \sum_{q\in \Sigma }\left[ \int
u(\xi )\delta _{q}(\xi )d\xi \right] \sigma _{q}(x)v(x)dx=\int u(\xi )v(\xi
)d\xi .
\end{equation*}%
We have that $v(x)=\sum_{q\in \Sigma }v_{q}\delta _{q}(x)$ with $v_{q}=\int
\sigma _{q}(x)v(x)dx;$ then 
\begin{eqnarray*}
\ \int \sum_{q\in \Sigma }\left[ \int u(\xi )\delta _{q}(\xi )d\xi \right]
\sigma _{q}(x)v(x)dx &=&\sum_{q\in \Sigma }\left( \int u(\xi )\delta
_{q}(\xi )d\xi \right) \left( \int \sigma _{q}(x)v(x)dx\right) = \\
\sum_{q\in \Sigma }\left( \int u(\xi )\delta _{q}(\xi )d\xi \right) v_{q}
&=&\int u(\xi )\left[ \sum_{q\in \Sigma }v_{q}\delta _{q}(\xi )\right] d\xi
=\int u(\xi )v(\xi )d\xi .
\end{eqnarray*}

The other equalities can be proved similarly. \end{proof}

In particular, if $f\in L_{loc}^{1}(\mathbb{R}),$ the function $\widetilde{%
f^{\ast }}$ is well defined. From now on we will simplify the notation just
writing $\widetilde{f}.$\bigskip

\begin{example}
\label{esempio1} Take $|x|^{-1/2}\in L_{loc}^{1}(-1,1)$, then 
\begin{equation*}
\widetilde{|x|^{-1/2}}=\sum_{q\in \Sigma }\left( \int^{\ast }|\xi
|^{-1/2}\delta _{q}(\xi )d\xi \right) \sigma _{q}(x)
\end{equation*}%
makes sense for every $x\in \mathbb{R}^{\ast };$ in particular 
\begin{equation*}
\left( \widetilde{|x|^{-1/2}}\right) _{x=0}=\int^{\ast }|x|^{-1/2}\delta
_{0}(x)dx,
\end{equation*}%
and it is easy to check that this is an infinite number. Notice that the
ultrafunction $\widetilde{|x|^{-1/2}}$ is different from $\left(
|x|^{-1/2}\right) ^{\ast }$ since the latter is not defined for $x=0\ ($and
they also differ for $|x|>\beta ).$
\end{example}

Now we want to show some interesting relations between $\widetilde{f}\ $and $%
f^{\ast }.$ More precisely we are interested in the following question.

Take $f\in L_{loc}^{1}(\mathbb{R})$ and $\Omega \subset \mathbb{R}$ $;$
which are the conditions that ensure the following:%
\begin{equation}
\forall x\in \Omega ^{\ast },\ \widetilde{f}(x)=f^{\ast }(x)?  \tag{Q}
\label{Q}
\end{equation}%
Notice that if $f\in L_{loc}^{1}(\mathbb{R}),$ $f$ and $f^{\ast }$ are not
defined pointwise and hence the above equality must be intended for almost
every $x.$

\begin{lemma}
Let $\Omega \subset \mathbb{R}$ be an open set and let $f\in L_{loc}^{1}(%
\mathbb{R})\mathfrak{.}$ Then 
\begin{equation*}
\forall ^{a.e.}x\in \Omega \ \ f(x)=0\Leftrightarrow \forall x\in \Omega
^{\ast }\cap \lbrack -\beta ,\beta ]\ \ \widetilde{f}(x)=0.
\end{equation*}
\end{lemma}

\begin{proof} We recall that, by (\ref{bella}), \textbf{\ }%
\begin{equation*}
\widetilde{f}(x)=\sum_{q\in \Sigma }\left[ \int f^{\ast }(\xi )\delta
_{q}(\xi )d\xi \right] \sigma _{q}(x).
\end{equation*}

If $\forall ^{a.e.}x\in \Omega ,\ f(x)=0$ then by Leibnitz principle we have
that $\forall ^{a.e.}x\in \Omega ^{\ast },\ f^{\ast }(x)=0$, so that $%
\forall x\in \Omega ^{\ast }\cap \lbrack -\beta ,\beta ],\ \widetilde{f}%
(x)=0 $ follows by (\ref{bella}).\\
Conversely, let us suppose that there is an open bounded interval $%
I\subseteq \Omega $ such that $\forall ^{a.e.}x\in I$ $f(x)\neq 0$ (we
suppose that $\forall ^{a.e.}x\in I$ $f(x)>0).$ By Leibnitz principle, we
have that $\forall ^{a.e.}x\in I^{\ast }$ $f^{\ast }(x)>0.$ Let $q\in 
\mathbb{I}_{j\text{ }}.$ Then%
\begin{eqnarray*}
0 &<&\int_{\mathbb{I}_{j\text{ }}}^{\ast }f^{\ast }(x)\chi
_{j}(x)dx=\int^{\ast }f^{\ast }(x)\chi _{j}(x)dx= \\
&=&\int^{\ast }\widetilde{f}(x)\chi _{j}(x)dx=0,
\end{eqnarray*}

since by hypothesis $\ \widetilde{f}(x)=0$ $\forall x\in \Omega ^{\ast }\cap
\lbrack -\beta ,\beta ]\supset I^{\ast }.$ And this is clearly absurd. \end{proof}

\begin{corollary}
\label{pluto}Let $\Omega \subset \mathbb{R}$ be an open set and let $f,g\in
L_{loc}^{1}(\mathbb{R})\mathfrak{;}$ then 
\begin{equation*}
\forall ^{a.e.}x\in \Omega \ \ f(x)=g(x)\Leftrightarrow \forall x\in \Omega
^{\ast }\cap \lbrack -\beta ,\beta ]\ \text{\ }\widetilde{f}(x)=\widetilde{g}%
(x).
\end{equation*}
\end{corollary}

\begin{proof} This follows immediatly by applying the previous theorem to
the function $h(x)=f(x)-g(x),$ since the operation $f\rightarrow \widetilde{f%
}$ is linear$.$ \end{proof}

\begin{theorem}
\label{locale}Let $\Omega \subset \mathbb{R}$ be an open bounded set, let $%
f\in L_{loc}^{1}(\mathbb{R})\mathfrak{;}$ if $f|_{\Omega }\in \mathcal{C}%
^{1}\left( \Omega \right) $ then 
\begin{equation*}
\forall x\in \Omega ^{\ast }\cap \lbrack -\beta ,\beta ]\text{ }\ \widetilde{%
f}(x)=f^{\ast }(x).
\end{equation*}
\end{theorem}

\begin{proof} Let $\left\{ \delta _{a}(x)\right\} _{a\in \Sigma }$ be a
Delta basis, let $y\in \Omega ^{\ast }$ and let $y\in \mathbb{I}_{j\text{ }%
}. $ Since, by (\ref{tbase}), for every $q\in \Sigma $ with $q\notin \mathbb{%
I}_{j\text{ }}\sigma _{q}(y)=0,$ by (\ref{bella}) we deduce that 
\begin{equation*}
\widetilde{f}(y)=\sum_{q\in \Sigma \cap \mathbb{I}_{j\text{ }}}\left[ \int_{%
\mathbb{I}_{j\text{ }}}f^{\ast }(\xi )\delta _{q}(\xi )d\xi \right] \sigma
_{q}(y).
\end{equation*}

Now let $g_{j}(x)$ be the function such that%
\begin{equation*}
g_{j}(x)=\left\{ 
\begin{array}{cc}
f^{\ast }(x) & \text{if }x\in \mathbb{I}_{j\text{ }}; \\ 
0 & \text{otherwise.}%
\end{array}%
\right.
\end{equation*}

Since $f|_{\Omega }\in \mathcal{C}^{1}\left( \Omega \right) $ then $g_{j}(x)$
is an ultrafunction. By construction, we have that $g_{j}(y)=\widetilde{f}%
(y) $ since, by (\ref{brava+}), 
\begin{eqnarray*}
g_{j}(y) &=&\sum_{q\in \Sigma }g_{j}(q)\sigma _{q}(y)=\sum_{q\in \Sigma \cap 
\mathbb{I}_{j\text{ }}}\left[ \int_{\mathbb{I}_{j\text{ }}}g_{j}(\xi )\delta
_{q}(\xi )d\xi \right] \sigma _{q}(y)= \\
&=&\sum_{q\in \Sigma \cap \mathbb{I}_{j\text{ }}}\left[ \int_{\mathbb{I}_{j%
\text{ }}}f^{\ast }(\xi )\delta _{q}(\xi )d\xi \right] \sigma _{q}(y)=%
\widetilde{f}(y).
\end{eqnarray*}

But, by definition, $g_{j}(y)=f^{\ast }(y);$ hence we deduce that $f^{\ast
}(y)=\widetilde{f}(y).$ \end{proof}

\begin{example}
If $f(x)=1,\ $then%
\begin{equation*}
\widetilde{1}(x)=\left\{ 
\begin{array}{cc}
1 & if\ x\in \left[ -\beta ,\beta \right] _{\mathbb{R}^{\ast }}; \\ 
0 & if\ x\notin \left[ -\beta ,\beta \right] _{\mathbb{R}^{\ast }}.%
\end{array}%
\right.
\end{equation*}
\end{example}

By Theorem \ref{locale} and the above example, we get:

\begin{corollary}
\label{yuo}Let $f\in \mathcal{C}^{1}\left( \Omega \right) $; then, 
\begin{equation*}
\widetilde{f}=f^{\ast }\cdot \widetilde{1}
\end{equation*}
\end{corollary}

By Theorem \ref{locale}, given a function $f(x)\in \mathcal{C}^{1}\left( 
\mathbb{R}\right) $ we have that $\widetilde{f}(x)$ extends $f(x)$ to $\left[
-\beta ,\beta \right] _{\mathbb{R}^{\ast }}$. $\widetilde{f}(x)$ will be
called the \textbf{canonical extension of }$f(x).$ With some abuse of
notation, $\widetilde{f}(x)$ will be called the "canonical extension of $%
f(x) $" even when $f(x)\in L_{loc}^{1}(\mathbb{R}).$

\begin{example}
If we consider the Example \ref{esempio1}, by Theorem \ref{locale}, we have
that 
\begin{equation*}
\forall a\in \lbrack -\beta ,\beta ]\backslash \mathfrak{mon(}0),\ \left( 
\widetilde{|x|^{-1/2}}\right) _{x=a}=\left( |x|^{-1/2}\right) _{x=a}^{\ast
}=|a|^{-1/2}.
\end{equation*}
\end{example}

\begin{example}
For a fixed $k\in \mathbb{R},$ the function $e^{ikx}$ defines a unique
ultrafunction $\widetilde{e^{ikx}}.$ Notice that $\widetilde{e^{ikx}}$ is
different from the natural extension of $e^{ikx}$ even if%
\begin{equation*}
\forall x\in \mathfrak{gal}\left( 0\right) ,\ \widetilde{e^{ikx}}=e^{ikx}.
\end{equation*}
\end{example}

\section{Derivative\label{D}}

\begin{definition}
For every ultrafunction $u\in \mathfrak{U}(\mathbb{R}),$ the derivative $%
Du(x)$ of $u(x)$ is the ultrafunction defined by the following formula:%
\begin{equation}
Du(x)=P_{\mathfrak{U}}u^{\prime }+\dsum\limits_{j=1}^{\ell -1}\triangle
u(\gamma _{j})\delta _{\gamma _{j}}(x),  \label{marina}
\end{equation}%
where $P_{\mathfrak{U}}u^{\prime }$ denotes the orthogonal projection of $%
u^{\prime }$ on $\mathfrak{U}(\mathbb{R})$ w.r.t. the $L^{2}$ scalar product
and, for every $j=1,...,l-1$,%
\begin{equation*}
\triangle u(\gamma _{j})=u^{+}(\gamma _{j})-u^{-}(\gamma _{j}).
\end{equation*}
\end{definition}

\begin{theorem}
For every $u,v\in \mathfrak{U}(\mathbb{R})$ the following equality holds:%
\begin{equation}
\int Du(x)v(x)\ dx=-\int u(x)Dv(x)\ dx+\left[ u(x)v(x)\right] _{-\beta
}^{\beta }.  \label{manola}
\end{equation}
\end{theorem}

\begin{proof} We have:%
\begin{equation*}
\int (Du(x)v(x)+u(x)Dv(x))dx=
\end{equation*}%
\begin{equation*}
\int \left( P_{V}u^{\prime }(x)+\dsum\limits_{j=1}^{\ell -1}\triangle
u(\gamma _{j})\delta _{\gamma _{j}}(x)\right) v(x)dx+\int \left(
P_{V}v^{\prime }(x)+\dsum\limits_{j=1}^{\ell -1}\triangle v(\gamma
_{j})\delta _{\gamma _{j}}(x)\right) u(x)dx=
\end{equation*}

\begin{equation*}
\int \left[ P_{V}u^{\prime }(x)v(x)+u(x)P_{V}v^{\prime }(x)\right] dx+
\end{equation*}%
\begin{equation*}
\int \left[ \left( \dsum\limits_{j=0}^{\ell -1}\triangle u(\gamma
_{j})\delta _{\gamma _{j}}(x)\right) v(x)+\left( \dsum\limits_{j=0}^{\ell
-1}\triangle v(\gamma _{j})\delta _{\gamma _{j}}(x)\right) u(x)\right] dx=
\end{equation*}%
\begin{equation*}
\int \left[ P_{V}u^{\prime }(x)v(x)+u(x)P_{V}v^{\prime }(x)\right]
dx+\dsum\limits_{j=1}^{\ell -1}\left[ \triangle u(\gamma _{j})v(\gamma
_{j})+\triangle v(\gamma _{j})u(\gamma _{j})\right] .
\end{equation*}%
Now let us compute the two terms of the sum separately; the first one:%
\begin{equation*}
\int \left[ P_{V}u^{\prime }(x)v(x)+u(x)P_{V}v^{\prime }(x)\right]
dx=\dsum\limits_{j=0}^{\ell -1}\int_{\gamma _{j}}^{\gamma _{j+1}}\left[
P_{V}u^{\prime }(x)v(x)+u(x)P_{V}v^{\prime }(x)\right] dx=
\end{equation*}%
\begin{equation*}
\dsum\limits_{j=0}^{\ell -1}\int_{\gamma _{j}}^{\gamma _{j+1}}\left[
u^{\prime }(x)v(x)+u(x)v^{\prime }(x)\right] dx=\dsum\limits_{j=0}^{\ell
-1}\int_{\gamma _{j}}^{\gamma _{j+1}}\left( u(x)v(x)\right) ^{\prime }dx=
\end{equation*}%
\begin{equation*}
=\dsum\limits_{j=0}^{\ell -1}\left[ u^{-}(\gamma _{j+1})v^{-}(\gamma
_{j+1})-u^{+}(\gamma _{j})v^{+}(\gamma _{j})\right] .
\end{equation*}

The second one:%
\begin{equation*}
\dsum\limits_{j=1}^{\ell -1}\left[ \triangle u(\gamma _{j})v(\gamma
_{j})+\triangle v(\gamma _{j})u(\gamma _{j})\right] =
\end{equation*}%
\begin{equation*}
\dsum\limits_{j=1}^{\ell -1}\left( (u^{+}(\gamma _{j})-u^{-}(\gamma
_{j}))\left( \frac{v^{+}(\gamma _{j})+v^{-}(\gamma _{j})}{2}\right)
+(v^{+}(\gamma _{j})-v^{-}(\gamma _{j}))\left( \frac{u^{+}(\gamma
_{j})+u^{-}(\gamma _{j})}{2}\right) \right) =
\end{equation*}%
\begin{equation*}
\dsum\limits_{j=1}^{\ell -1}(u^{+}(\gamma _{j})v^{+}(\gamma
_{j})-u^{-}(\gamma _{j})v^{-}(\gamma _{j})).
\end{equation*}

Thus 
\begin{equation*}
\int \left[ P_{V}u^{\prime }v(x)+u(x)P_{V}v^{\prime }\right]
dx+\dsum\limits_{j=0}^{\ell -1}\left( \triangle u(\gamma _{j})v(\gamma
_{j})+\triangle v(\gamma _{j})u(\gamma _{j})\right) =
\end{equation*}%
\begin{equation*}
\dsum\limits_{j=0}^{\ell -1}\left( u^{-}(\gamma _{j+1})v^{-}(\gamma
_{j+1})-u^{+}(\gamma _{j})v^{+}(\gamma _{j})\right)
+\dsum\limits_{j=1}^{\ell -1}(u^{+}(\gamma _{j})v^{+}(\gamma
_{j})-u^{-}(\gamma _{j})v^{-}(\gamma _{j})).
\end{equation*}%
But $\dsum\limits_{j=0}^{\ell -1}\left( u^{-}(\gamma _{j+1})v^{-}(\gamma
_{j+1})-u^{+}(\gamma _{j})v^{+}(\gamma _{j})\right) =-u(-\beta )v(-\beta
)+u(\beta )v(\beta )+\dsum\limits_{j=1}^{\ell -1}\left( u^{-}(\gamma
_{j})v^{-}(\gamma _{j})-u^{+}(\gamma _{j})v^{+}(\gamma _{j})\right) $, hence%
\begin{equation*}
\dsum\limits_{j=0}^{\ell -1}\left( u^{-}(\gamma _{j+1})v^{-}(\gamma
_{j+1})-u^{+}(\gamma _{j})v^{+}(\gamma _{j})\right)
+\dsum\limits_{j=1}^{\ell -1}(u^{+}(\gamma _{j})v^{+}(\gamma
_{j})-u^{-}(\gamma _{j})v^{-}(\gamma _{j}))=
\end{equation*}%
\begin{equation*}
u(\beta )v(\beta )-u(-\beta )v(-\beta ).\qedhere
\end{equation*}

\end{proof}

\begin{remark}
\label{zurlino}The generalized derivative 
\begin{equation*}
D:\mathfrak{U}(\mathbb{R})\rightarrow \mathfrak{U}(\mathbb{R})
\end{equation*}%
is a linear operator, as can be directly derived by $\left( \ref{marina}%
\right) .$ Moreover for every ultrafunction $u\in \mathfrak{U}(\mathbb{R}%
)\cap \mathcal{C}^{1}(\mathbb{R})^{\ast }$ we have that%
\begin{equation}
Du(x)=\widetilde{u^{\prime }(x)},  \label{carla}
\end{equation}%
since in this case $\triangle u(\gamma _{j})=0$ for every $j=1,...,l-1$. In
particular, if $f\in \mathcal{C}^{2}(\mathbb{R})$ then $\forall x\in \lbrack
-\beta ,\beta ]$%
\begin{equation}
Df^{\ast }(x)=\left( f^{\prime }\right) ^{\ast }(x),  \label{carlino}
\end{equation}%
because in this case $\left( f^{\prime }\right) ^{\ast }(x)\in \mathfrak{U}(%
\mathbb{R}),$ so $P_{\mathfrak{U}}\left( f^{\prime }\right) ^{\ast }=\left(
f^{\prime }\right) ^{\ast }.$
\end{remark}

\begin{remark}
Notice that by (\ref{carla}) and (\ref{carlino}) we have that $\forall f\in 
\mathcal{C}^{1}(\mathbb{R})$ and $\forall x\in \mathbb{R}$%
\begin{equation*}
D\widetilde{f}(x)=\widetilde{f^{\prime }(x)}\sim f^{\prime }(x)
\end{equation*}%
and $\forall f\in \mathcal{C}^{2}(\mathbb{R})$ and $\forall x\in \mathbb{R}$%
\begin{equation*}
D\widetilde{f}(x)=f^{\prime }(x).
\end{equation*}%
In this sense, $D$ extends the usual derivative to all ultrafunctions and to
all the points in $\mathbb{R}^{\ast }.$
\end{remark}

\textbf{Example 1: }By (\ref{manola}) we have that%
\begin{equation}
D\widetilde{1}=0.  \label{clara}
\end{equation}%
If $u(x)=\widetilde{x}$ then 
\begin{equation*}
D\widetilde{x}=\widetilde{1}.
\end{equation*}

\textbf{Example 2: }If $a\neq -\beta ,b\neq \beta $ and $u(x)=\chi _{\left[
a,b\right] }(x),$ then%
\begin{equation*}
D\chi _{\left[ a,b\right] }=\delta _{a}-\delta _{b}.
\end{equation*}

\textbf{Example 3: }If $a=-\beta ,b\neq \beta $ and $u(x)=\chi _{\left[ a,b%
\right] }(x),$ then%
\begin{equation*}
D\chi _{\left[ a,b\right] }=-\delta _{b},
\end{equation*}

and if $a\neq -\beta ,b=\beta $ and $u(x)=\chi _{\left[ a,b\right] }(x),$
then%
\begin{equation*}
D\chi _{\left[ a,b\right] }=\delta _{a}.
\end{equation*}

\textbf{Example 4: }$u(x)=w(x)\chi _{\left[ a,b\right] }(x)\ $with $a,b\in
\Gamma \backslash \left\{ -\beta ,\beta \right\} ,$ then, by (\ref{marina})%
\begin{equation*}
u(x)^{\prime }=P_{V}w^{\prime }(x)\chi _{\left[ a,b\right] }(x)+w(a)\delta
_{a}(x)-w(b)\delta _{b}(x).
\end{equation*}

\section{Definite integral}

Since every ultrafunction is an internal function, the definite integral is
well defined:%
\begin{equation*}
\int_{a}^{b}u(x)dx:=\left( \int_{a}^{b}\right) ^{\ast }u(x)dx.
\end{equation*}

Let us observe that, for every $a,b\in \Gamma $, the characteristic function 
$\chi _{\left[ a,b\right] }$ of $[a,b]$ in the usual sense and the
characteristic function $\chi _{\left[ a,b\right] _{\mathbb{R}^{\ast }}}$ of 
$[a,b]$ in the sense of ultrafunctions are different (at most)\ only in the
points $a$ and $b$. In particular, for every ultrafunction $u(x)$ we have 
\begin{equation*}
\int_{a}^{b}u(x)dx=\int^{\ast }u(x)\chi _{\left[ a,b\right]
}(x)dx=\int^{\ast }u(x)\chi _{\left[ a,b\right] _{\mathbb{R}^{\ast }}}(x)dx.
\end{equation*}%
This observation is important to prove the following theorem:

\begin{corollary}
\label{katia}(\textbf{Fundamental Theorem of Calculus}) If $a,b\in \Gamma ,$
then%
\begin{equation*}
\int_{a}^{b}Du(x)dx=u(b)-u(a).
\end{equation*}
\end{corollary}

\begin{proof} We have: 
\begin{eqnarray*}
\int_{a}^{b}Du(x)dx &=&\int^{\ast }Du(x)\chi _{\left[ a,b\right] }(x)dx= \\
\int^{\ast }Du(x)\chi _{\left[ a,b\right] _{\mathbb{R}^{\ast }}}(x)dx
&=&-\int u(x)D\chi _{\left[ a,b\right] _{\mathbb{R}^{\ast }}}(x)dx+\left[
u(x)\chi _{\left[ a,b\right] _{\mathbb{R}^{\ast }}}\right] _{-\beta }^{\beta
}.
\end{eqnarray*}

Now if $a\neq -\beta ,b\neq \beta $ we have $\left[ u(x)\chi _{\left[ a,b%
\right] _{\mathbb{R}^{\ast }}}\right] _{-\beta }^{\beta }=0$ and $D\chi _{%
\left[ a,b\right] _{\mathbb{R}^{\ast }}}(x)=\delta _{a}-\delta _{b},$ so 
\begin{eqnarray*}
-\int u(x)D\chi _{\left[ a,b\right] _{\mathbb{R}^{\ast }}}(x)dx &=&-\int
u(x)(\delta _{a}-\delta _{b})dx= \\
&&u(b)-u(a).
\end{eqnarray*}

If $a=-\beta ,b\neq \beta $ we have $\left[ u(x)\chi _{\left[ a,b\right] _{%
\mathbb{R}^{\ast }}}\right] _{-\beta }^{\beta }=-u(-\beta )$ and $D\chi _{%
\left[ a,b\right] }(x)=-\delta _{b},$ so 
\begin{eqnarray*}
-\int u(x)D\chi _{\left[ a,b\right] _{\mathbb{R}^{\ast }}}(x)dx-u(-\beta )
&=&-\int u(x)(-\delta _{b})dx-u(-\beta )= \\
u(b)-u(-\beta ) &=&u(b)-u(a).
\end{eqnarray*}

The case $a\neq -\beta ,b=\beta $ can be proved similarly. If $a=-\beta
,b=\beta $ then%
\begin{eqnarray*}
\int Du(x)\chi _{\left[ -\beta ,\beta \right] }(x)dx &=&\int Du(x)\widetilde{%
1}dx= \\
-\int u(x)D\widetilde{1}dx+[u(x)]_{-\beta }^{\beta } &=&u(\beta )-u(-\beta ),
\end{eqnarray*}%
since $D\widetilde{1}=0.$ \end{proof}

Notice that $\mathbb{R}\subset \Gamma ;$ thus if $f\in \mathcal{C}^{1}(%
\mathbb{R})\ $we have that, $\forall a,b\in \mathbb{R}$,%
\begin{equation*}
\int_{a}^{b}D\widetilde{f}(x)dx=f(b)-f(a).
\end{equation*}%
A question that arises is: does it hold, for ultrafunctions, some kind of
"rule of integration by parts for continuous functions", at least for the
points in $\Gamma $? E.g., is it true that, if $u,v\in \mathfrak{U}(\mathbb{R%
})$ and $a,b\in \Gamma ,$ then 
\begin{equation}
\int_{a}^{b}Du(x)v(x)\ dx=-\int_{a}^{b}u(x)Dv(x)\ dx+\left[ u(x)v(x)\right]
_{a}^{b}?  \label{byotyny}
\end{equation}%
The answer is no, as a simple computation shows. Nevertheless, we have the
following:

\begin{prop}
Let $u,v\in \mathfrak{U}(\mathbb{R})\cap \mathcal{C}^{1}(\mathbb{R})^{\ast
}, $ and $\gamma _{n}<\gamma _{m}\in \Gamma .$ Then%
\begin{equation*}
\int_{\gamma _{n}}^{\gamma _{m}}Du(x)v(x)\ dx=-\int_{\gamma _{n}}^{\gamma
_{m}}u(x)Dv(x)\ dx+u^{-}(\gamma _{m})v^{-}(\gamma _{m})-u^{+}(\gamma
_{n})v^{+}(\gamma _{n}).
\end{equation*}
\end{prop}

\begin{proof} By (\ref{carla}), since $u,v\in \mathfrak{U}(\mathbb{R})\cap 
\mathcal{C}^{1}(\mathbb{R})^{\ast }$ then $Du=\widetilde{u^{\prime }}$ and $%
Dv=\widetilde{v^{\prime }}.$ Moreover, since $\mathfrak{U}(\mathbb{R}%
)=\bigoplus_{j=0}^{l-1}\mathbb{I}_{j},$ if for every $j=0,...,l-1$ we denote
by $P_{j}$ the orthogonal projection on $\mathbb{I}_{j}$ we have%
\begin{equation*}
P_{\mathfrak{U}}u^{\prime }(x)=\sum_{j=0}^{l-1}P_{j}(u^{\prime }(x)).
\end{equation*}%
Now, if $m=n+1,$ since $u$ and $v$ are continuous we have 
\begin{eqnarray*}
\int_{\gamma _{n}}^{\gamma _{m}}Du(x)v(x)\ dx &=&\int_{\gamma _{n}}^{\gamma
_{m}}P_{\mathfrak{U}}u^{\prime }(x)v(x)\ dx= \\
\int_{\gamma _{n}}^{\gamma _{m}}P_{n}u^{\prime }(x)v(x)\ dx &=&\int_{\gamma
_{n}}^{\gamma _{m}}u^{\prime }(x)v(x)\ dx=
\end{eqnarray*}%
\begin{equation*}
-\int_{\gamma _{n}}^{\gamma _{m}}u(x)v^{\prime }(x)dx+u^{-}(\gamma
_{m})v^{-}(\gamma _{m})-u^{+}(\gamma _{n})v^{+}(\gamma _{n})=
\end{equation*}%
\begin{equation*}
-\int_{\gamma _{n}}^{\gamma _{m}}u(x)P_{j}v^{\prime }(x)dx+u^{-}(\gamma
_{m})v^{-}(\gamma _{m})-u^{+}(\gamma _{n})v^{+}(\gamma _{n})=
\end{equation*}%
\begin{equation*}
-\int_{\gamma _{n}}^{\gamma _{m}}u(x)Dv(x)dx+u^{-}(\gamma _{m})v^{-}(\gamma
_{m})-u^{+}(\gamma _{n})v^{+}(\gamma _{n}).
\end{equation*}

In the general case, 
\begin{equation*}
\int_{\gamma _{n}}^{\gamma _{m}}Du(x)v(x)\ dx=\sum_{i=n}^{m-1}\int_{\gamma
_{i}}^{\gamma _{i+1}}Du(x)v(x)\ dx=
\end{equation*}%
\begin{equation*}
\sum_{i=n}^{m-1}\left[ -\int_{\gamma _{i}}^{\gamma
_{i+1}}u(x)Dv(x)dx+u^{-}(\gamma _{i+1})v^{-}(\gamma _{i+1})-u^{+}(\gamma
_{i})v^{+}(\gamma _{i})\right] ,
\end{equation*}

and since $u,v$ are continuous we have%
\begin{equation*}
\sum_{i=n}^{m-1}\left[ -\int_{\gamma _{i}}^{\gamma
_{i+1}}u(x)Dv(x)dx+u^{-}(\gamma _{i+1})v^{-}(\gamma _{i+1})-u^{+}(\gamma
_{i})v^{+}(\gamma _{i})\right] =
\end{equation*}%
\begin{equation*}
\sum_{i=n}^{m-1}\left[ -\int_{\gamma _{i}}^{\gamma _{i+1}}u(x)Dv(x)dx\right]
+u^{-}(\gamma _{m})v^{-}(\gamma _{m})-u^{+}(\gamma _{n})v^{+}(\gamma _{n})=
\end{equation*}%
\begin{equation*}
-\int_{\gamma _{n}}^{\gamma _{m}}u(x)Dv(x)\ dx+u^{-}(\gamma
_{m})v^{-}(\gamma _{m})-u^{+}(\gamma _{n})v^{+}(\gamma _{n}).\qedhere
\end{equation*}%
\end{proof}

The previous proposition is, in general, false if at least one between $u,v$
is not in $\mathcal{C}^{1}(\mathbb{R})^{\ast }.$ The reason is that, by
definition, the derivative has the following expression:%
\begin{equation*}
Du(x)=P_{\mathfrak{U}}u^{\prime }+\dsum\limits_{j=1}^{\ell -1}\triangle
u(\gamma _{j})\delta _{\gamma _{j}}(x),
\end{equation*}%
and the presence of $\dsum\limits_{j=1}^{\ell -1}\triangle u(\gamma
_{j})\delta _{\gamma _{j}}(x)$ is what makes (\ref{byotyny}) to be false.
Just for sake of completeness, we now show how to obtain a relaxed version of
(\ref{byotyny}) by considering a different possible notion of derivative on $%
\mathfrak{U}(\mathbb{R}).$ The relaxed version of (\ref{byotyny}) is the
following: since the functions in $\mathfrak{U}(\mathbb{R})$ are piecewise $%
\mathcal{C}^{1}$ functions, does it hold, for ultrafunctions, an analogue of
the rule of integration by parts for piecewise $\mathcal{C}^{1}$ functions?
Namely, is it true that, if $u,v\in \mathfrak{U}(\mathbb{R})$ and $\gamma
_{n}<\gamma _{m}\in \Gamma ,$ then%
\begin{equation}
\int_{\gamma _{n}}^{\gamma _{m}}Du(x)v(x)\ dx=-\int_{\gamma _{n}}^{\gamma
_{m}}u(x)Dv(x)\ dx+\sum_{i=n}^{m-1}\left[ u^{-}(\gamma _{i+1})v^{-}(\gamma
_{i+1})-u^{+}(\gamma _{i})v^{+}(\gamma _{i})\right] ?  \label{partii}
\end{equation}%
With the operator $D$ the answer is no. But there is a different linear
operator that actually satisfies $(\ref{partii}):$

\begin{definition}
We denote by $D_{2}u(x)$ the linear operator such that, for every $u\in 
\mathfrak{U}(\mathbb{R})$ , we have%
\begin{equation*}
D_{2}u(x)=P_{\mathfrak{U}}(u^{\prime }(x)).
\end{equation*}
\end{definition}

Since $\mathfrak{U}(\mathbb{R})=\bigoplus_{j=0}^{l-1}\mathbb{I}_{j},$ if we
denote by $P_{j}$ the orthogonal projection on $\mathbb{I}_{j},$ we have%
\begin{equation*}
D_{2}u(x)=P_{\mathfrak{U}}u^{\prime }(x)=\sum_{j=0}^{l-1}P_{j}(u^{\prime
}(x)).
\end{equation*}%
Moreover we have that, if $u(x)$ is continuous in $\gamma _{j},\gamma _{j+1}$%
, then 
\begin{equation*}
Du(x)=D_{2}u(x)
\end{equation*}%
on $\mathbb{I}_{j}.$ In particular, if $u(x)$ is continuous in $[-\beta
,\beta ]$ then 
\begin{equation*}
Du(x)=D_{2}u(x).
\end{equation*}%
This new linear operator is what we need to obtain the generalization to $%
\mathfrak{U}(\mathbb{R})$ of the rule of integration by parts for piecewise
continuous functions:

\begin{theorem}
\textbf{(Integration by parts for piecewise }$\mathcal{C}^{1}$\textbf{\
functions) }For every $u,v\in \mathfrak{U}(\mathbb{R})$ and $\gamma
_{n}<\gamma _{m}\in \Gamma $ we have 
\begin{equation*}
\int_{\gamma _{n}}^{\gamma _{m}}D_{2}u(x)v(x)\ dx=-\int_{\gamma
_{n}}^{\gamma _{m}}u(x)D_{2}v(x)\ dx+\sum_{i=n}^{m-1}\left[ u^{-}(\gamma
_{i+1})v^{-}(\gamma _{i+1})-u^{+}(\gamma _{i})v^{+}(\gamma _{i})\right] .
\end{equation*}
\end{theorem}

\begin{proof} If $m=n+1$ then 
\begin{equation*}
\int_{\gamma _{n}}^{\gamma _{m}}D_{2}u(x)v(x)\ dx=\int_{\gamma _{n}}^{\gamma
_{m}}u^{\prime }(x)v(x)\ dx=
\end{equation*}%
\begin{equation*}
-\int_{\gamma _{n}}^{\gamma _{m}}u(x)v^{\prime }(x)dx+u^{-}(\gamma
_{m})v^{-}(\gamma _{m})-u^{+}(\gamma _{n})v^{+}(\gamma _{n})=
\end{equation*}%
\begin{equation*}
-\int_{\gamma _{n}}^{\gamma _{m}}u(x)D_{2}v(x)dx+u^{-}(\gamma
_{m})v^{-}(\gamma _{m})-u^{+}(\gamma _{n})v^{+}(\gamma _{n}).
\end{equation*}%
In the general case we have%
\begin{equation*}
\int_{\gamma _{n}}^{\gamma _{m}}D_{2}u(x)v(x)\
dx=\sum_{i=n}^{m-1}\int_{\gamma _{i}}^{\gamma _{i+1}}D_{2}u(x)v(x)\ dx=
\end{equation*}%
\begin{equation*}
\sum_{i=n}^{m-1}\left( -\int_{\gamma _{i}}^{\gamma
_{i+1}}u(x)D_{2}v(x)dx+u^{-}(\gamma _{i+1})v^{-}(\gamma _{i+1})-u^{+}(\gamma
_{i})v^{+}(\gamma _{i})\right) =
\end{equation*}%
\begin{equation*}
-\int_{\gamma _{n}}^{\gamma _{m}}u(x)D_{2}v(x)\ dx+\sum_{i=n}^{m-1}\left[
u^{-}(\gamma _{i+1})v^{-}(\gamma _{i+1})-u^{+}(\gamma _{i})v^{+}(\gamma _{i})%
\right] .\qedhere
\end{equation*}%
\end{proof}

In particular, since $D_{2}\widetilde{1}=0,$ it is immediate to prove that
the following holds:

\begin{corollary}
(\textbf{Fundamental Theorem of Calculus for piecewise continuous functions}%
) For every $u\in \mathfrak{U}(\mathbb{R})$ and $\gamma _{n}<\gamma _{m}\in
\Gamma $ we have 
\begin{equation*}
\int_{\gamma _{n}}^{\gamma _{m}}D_{2}u(x)dx=\sum_{i=n}^{m-1}\left[
u^{-}(\gamma _{i+1})-u^{+}(\gamma _{i})\right] .
\end{equation*}
\end{corollary}

Of course, the derivative $D_{2}$ has also many drawbacks, e.g. for every
grid function $g$ we have $D_{2}(g)=0$. So in the following we will only
consider the derivative $D$.

\section{Ultrafunctions and distributions}

In this section we briefly explain how to associate an ultrafunction to
every distribution $T\in \mathcal{C}^{-\infty }\left( \mathbb{R}\right) $,
where%
\begin{equation*}
\mathcal{C}^{-\infty }\left( \mathbb{R}\right) =\{T\in \mathcal{D}^{\prime }(%
\mathbb{R})\mid \exists k\in \mathbb{N}\text{, }\exists f\in \mathcal{C}%
^{0}\left( \mathbb{R}\right) \text{ such that }T=\partial ^{k}f\}.
\end{equation*}
Note that, by definition, if $T\in \mathcal{C}^{-\infty }\left( \mathbb{R}%
\right) $ then there exists a natural number $k$ and a function $f\in 
\mathcal{C}^{1}\left( \mathbb{R}\right) $ such that:%
\begin{equation}
T=\partial ^{k}f.
\end{equation}%
So it is natural to introduce the following definition:

\begin{definition}
\label{distr}Given a distribution $T\in \mathcal{C}^{-\infty }\left( \mathbb{%
R}\right) ,$ let $k$ be the minimum natural number such that there exists $%
f\in \mathcal{C}^{1}\left( \mathbb{R}\right) $ with $T=\partial ^{k}f.$ We
denote by $\widetilde{T}$ the ultrafunction 
\begin{equation*}
\widetilde{T}(x)=D^{k}f^{\ast }.
\end{equation*}
$\widetilde{T}$ will be called the ultrafunction associated with the
distribution $T$.
\end{definition}

\begin{proposition}
For every distribution $T\in \mathcal{C}^{-\infty }\left( \mathbb{R}\right)
, $ for every test function $\varphi \in \mathcal{D}\left( \mathbb{R}\right) 
$ we have 
\begin{equation*}
\int^{\ast }\widetilde{T}(x)\varphi ^{\ast }(x)dx=\left\langle T,\varphi
\right\rangle .
\end{equation*}
\end{proposition}

\begin{proof} Let us suppose that $T=\partial ^{k}f,$ where $k,f$ are
given as in Definition \ref{distr}. Then, by (\ref{manola}), since $\varphi
^{\ast }(\beta )=\varphi ^{\ast }(-\beta )=0,$ we have that 
\begin{eqnarray*}
\int^{\ast }\widetilde{T}(x)\varphi ^{\ast }(x)dx &=&\int^{\ast
}D^{k}f^{\ast }(x)\varphi ^{\ast }(x)dx=\left( -1\right) ^{k}\int^{\ast
}f^{\ast }(x)\partial ^{k}\varphi ^{\ast }(x)dx \\
&=&\left[ \left( -1\right) ^{k}\int f(x)\partial ^{k}\varphi (x)dx\right]
^{\ast }=\left\langle T,\varphi \right\rangle ^{\ast }=\left\langle
T,\varphi \right\rangle .\qedhere\end{eqnarray*}\end{proof}

In the forthcoming paper \cite{algebra} we will show that, actually, it is
possible to define an embedding of$\mathcal{\ }$the whole space of
distributions in a particular space of ultrafunctions; this definition will
be used to construct a particular algebra, related to ultrafunctions, in
which the distributions can be embedded.

\section{APPENDIX - $\Lambda $-theory\label{lt}}

In this section we present the basic notions of Non Archimedean Mathematics
and of Nonstandard Analysis following a method inspired by \cite{BDN2003}
(see also \cite{ultra} and \cite{belu2012}).

\subsection{Non Archimedean Fields\label{naf}}

Here, we recall the basic definitions and facts regarding Non Archimedean
fields. In the following, ${\mathbb{K}}$ will denote an ordered field. We
recall that such a field contains (a copy of) the rational numbers. Its
elements will be called numbers.

\begin{definition}
Let $\mathbb{K}$ be an ordered field. Let $\xi \in \mathbb{K}$. We say that:

\begin{itemize}
\item $\xi $ is infinitesimal if, for all positive $n\in \mathbb{N}$, $|\xi
|<\frac{1}{n}$;

\item $\xi $ is finite if there exists $n\in \mathbb{N}$ such as $|\xi |<n$;

\item $\xi $ is infinite if, for all $n\in \mathbb{N}$, $|\xi |>n$
(equivalently, if $\xi $ is not finite).
\end{itemize}
\end{definition}

\begin{definition}
An ordered field $\mathbb{K}$ is called Non-Archimedean if it contains an
infinitesimal $\xi \neq 0$.
\end{definition}

It is easily seen that all infinitesimal are finite, that the inverse of an
infinite number is a nonzero infinitesimal number, and that the inverse of a
nonzero infinitesimal number is infinite.

\begin{definition}
A superreal field is an ordered field $\mathbb{K}$ that properly extends $%
\mathbb{R}$.
\end{definition}

It is easy to show, due to the completeness of $\mathbb{R}$, that there are
nonzero infinitesimal numbers and infinite numbers in any superreal field.
Infinitesimal numbers can be used to formalize a new notion of "closeness":

\begin{definition}
\label{def infinite closeness} We say that two numbers $\xi, \zeta \in {%
\mathbb{K}}$ are infinitely close if $\xi -\zeta $ is infinitesimal. In this
case, we write $\xi \sim \zeta $.
\end{definition}

Clearly, the relation "$\sim $" of infinite closeness is an equivalence
relation.

\begin{theorem}
If $\mathbb{K}$ is a superreal field, every finite number $\xi \in \mathbb{K}
$ is infinitely close to a unique real number $r\sim \xi $, called the 
\textbf{shadow} or the \textbf{standard part} of $\xi $.
\end{theorem}

Given a finite number $\xi $, we denote its shadow as $sh(\xi )$, and we put 
$sh(\xi )=+\infty $ ($sh(\xi )=-\infty $) if $\xi \in \mathbb{K}$ is a
positive (negative) infinite number.\newline

\begin{definition}
Let $\mathbb{K}$ be a superreal field, and $\xi \in \mathbb{K}$ a number.
The \label{def monad} monad of $\xi $ is the set of all numbers that are
infinitely close to it:%
\begin{equation*}
\mathfrak{m}\mathfrak{o}\mathfrak{n}(\xi )=\{\zeta \in \mathbb{K}:\xi \sim
\zeta \},
\end{equation*}%
and the galaxy of $\xi $ is the set of all numbers that are finitely close
to it: 
\begin{equation*}
\mathfrak{gal}(\xi )=\{\zeta \in \mathbb{K}:\xi -\zeta \ \text{is\ finite}\}
\end{equation*}
\end{definition}

By definition, it follows that the set of infinitesimal numbers is $%
\mathfrak{mon}(0)$ and that the set of finite numbers is $\mathfrak{gal}(0)$.

\subsection{The $\Lambda $-limit\label{OL}}

In this section we will introduce a superreal field $\mathbb{K}$ and we will
analyze its main properties by mean of the $\Lambda $-theory (see also \cite%
{ultra}, \cite{belu2012}).

We set%
\begin{equation*}
\mathfrak{X}=\mathcal{P}_{fin}(\mathfrak{F(}\mathbb{R},\mathbb{R}));
\end{equation*}%
we will refer to $\mathfrak{X}$ as the "parameter space". Clearly $\left( 
\mathfrak{X},\subset \right) $ is a directed set and, as usual, a function $%
\varphi :\mathfrak{X}\rightarrow E$ will be called \textit{net }(with values
in $E$).\newline
We present axiomatically the notion of $\Lambda $-limit:

\bigskip {\Large Axioms of\ the }$\Lambda ${\Large -limit}

\begin{itemize}
\item \textsf{(}$\Lambda $-\textsf{1)}\ \textbf{Existence Axiom.}\ \textit{%
There is a superreal field} $\mathbb{K}\supset \mathbb{R}$ \textit{such that
every net }$\varphi :\mathfrak{X}\rightarrow \mathbb{R}$\textit{\ has a
unique limit }$L\in \mathbb{K}{\ }($\textit{called the} "$\Lambda $-limit" 
\textit{of}\emph{\ }$\varphi .)$ \textit{The} $\Lambda $-\textit{limit of }$%
\varphi $\textit{\ will be denoted as} 
\begin{equation*}
L=\lim_{\lambda \uparrow \Lambda }\varphi (\lambda ).
\end{equation*}%
\textit{Moreover we assume that every}\emph{\ }$\xi \in \mathbb{K}$\textit{\
is the }$\Lambda $-\textit{limit\ of some real function}\emph{\ }$\varphi :%
\mathfrak{X}\rightarrow \mathbb{R}$\emph{. }

\item ($\Lambda $-2)\ \textbf{Real numbers axiom}. \textit{If }$\varphi
(\lambda )$\textit{\ is} \textit{eventually} \textit{constant}, \textit{%
namely} $\exists \lambda _{0}\in \mathfrak{X},r\in \mathbb{R}$ such that $%
\forall \lambda \supset \lambda _{0},\ \varphi (\lambda )=r,$ \textit{then}%
\begin{equation*}
\lim_{\lambda \uparrow \Lambda }\varphi (\lambda )=r.
\end{equation*}

\item ($\Lambda $-3)\ \textbf{Sum and product Axiom}.\ \textit{For all }$%
\varphi ,\psi :\mathfrak{X}\rightarrow \mathbb{R}$\emph{: }%
\begin{eqnarray*}
\lim_{\lambda \uparrow \Lambda }\varphi (\lambda )+\lim_{\lambda \uparrow
\Lambda }\psi (\lambda ) &=&\lim_{\lambda \uparrow \Lambda }\left( \varphi
(\lambda )+\psi (\lambda )\right) ; \\
\lim_{\lambda \uparrow \Lambda }\varphi (\lambda )\cdot \lim_{\lambda
\uparrow \Lambda }\psi (\lambda ) &=&\lim_{\lambda \uparrow \Lambda }\left(
\varphi (\lambda )\cdot \psi (\lambda )\right) .
\end{eqnarray*}
\end{itemize}

\begin{theorem}
\label{brufolo}The set of axioms $\{$($\Lambda $-1)\textsf{,}($\Lambda $-2),(%
$\Lambda $-3)$\}$ is consistent.
\end{theorem}

Theorem \ref{brufolo} is proved in \cite{ultra} and in \cite{belu2013}.

The notion of $\Lambda $-limit can be extended to sets and functions in the
following way:

\begin{definition}
Let $E_{\lambda },$ $\lambda \in \mathfrak{X},$ be a family of sets. We
define%
\begin{equation*}
\lim_{\lambda \uparrow \Lambda }\ E_{\lambda }:=\left\{ \lim_{\lambda
\uparrow \Lambda }\psi (\lambda )\ |\ \psi (\lambda )\in E_{\lambda
}\right\} ;
\end{equation*}%
A set which is a $\Lambda $-\textit{limit\ is called \textbf{internal}.} In
particular, if $\forall \lambda \in \mathfrak{X,}$ $E_{\lambda }=E,$ we set $%
\lim_{\lambda \uparrow \Lambda }\ E_{\lambda }=E^{\ast },\ $namely%
\begin{equation*}
E^{\ast }:=\left\{ \lim_{\lambda \uparrow \Lambda }\psi (\lambda )\ |\ \psi
(\lambda )\in E\right\} .
\end{equation*}%
$E^{\ast }$ is called the \textbf{natural extension }of $E.$
\end{definition}

This definition, combined with axiom ($\Lambda $-1$)$, entails that 
\begin{equation*}
\mathbb{K}=\mathbb{R}^{\ast }.
\end{equation*}

\begin{definition}
Let 
\begin{equation*}
f_{\lambda }:\ E_{\lambda }\rightarrow \mathbb{R},\ \ \lambda \in \mathfrak{X%
},
\end{equation*}
be a family of functions; then we define a function%
\begin{equation*}
F:\left( \lim_{\lambda \uparrow \Lambda }\ E_{\lambda }\right) \rightarrow 
\mathbb{R}^{\ast }
\end{equation*}%
as follows%
\begin{equation*}
\lim_{\lambda \uparrow \Lambda }\ f_{\lambda }(\xi ):=f\left( \lim_{\lambda
\uparrow \Lambda }\psi (\lambda )\right) ;
\end{equation*}%
where $\psi (\lambda )$ is a net of numbers such that 
\begin{equation*}
\psi (\lambda )\in E_{\lambda }\ \ \text{and}\ \ \lim_{\lambda \uparrow
\Lambda }\psi (\lambda )=\xi
\end{equation*}%
A function which is a $\Lambda $-\textit{limit\ is called \textbf{internal}.}
In particular, if $\forall \lambda \in \mathfrak{X,}$ 
\begin{equation*}
f_{\lambda }=f,\ \ \ \ f:\ E\rightarrow \mathbb{R},
\end{equation*}%
we set 
\begin{equation*}
f^{\ast }=\lim_{\lambda \uparrow \Lambda }\ f_{\lambda }
\end{equation*}%
$f^{\ast }:E^{\ast }\rightarrow \mathbb{R}^{\ast }$ is called the \textbf{%
natural extension }of $f.$
\end{definition}

Notice that, while the \textit{\ }$\Lambda $-limit\ of a constant sequence
of numbers gives this number itself, the $\Lambda $-limit of a constant
sequence of sets is a larger set and the $\Lambda $-limit of a constant
sequence of functions is an extension of this function.

In a similar way it is possible to extend operator and functionals.

Finally, the $\Lambda $-limits satisfy the following important Theorem:

\begin{theorem}
\label{limit}\textbf{(Leibnitz Principle)} Let $S$ be a set, $\mathcal{R}$ a
relation defined on $S$ and $\varphi $,$\psi :\mathfrak{X}\rightarrow S$. If 
\begin{equation*}
\forall \lambda \in \mathfrak{X},\ \varphi (\lambda )\mathcal{R}\psi
(\lambda )
\end{equation*}%
then%
\begin{equation*}
\left( \underset{\lambda \uparrow {\mathbb{U}}}{\lim }\varphi (\lambda
)\right) \mathcal{R}^{\ast }\left( \underset{\lambda \uparrow {\mathbb{U}}}{%
\lim }\psi (\lambda )\right) .
\end{equation*}
\end{theorem}

\subsection{Hyperfinite sets and hyperfinite sums\label{HE}}

\begin{definition}
An internal set is called \textbf{hyperfinite} if it is the $\Lambda $-limit
of a net $\varphi :\mathfrak{X}\rightarrow \mathfrak{X}$.
\end{definition}

\begin{definition}
Given any set $E\in $ {$\mathbb{U}$}, the hyperfinite extension of $E$ is
defined as follows:%
\begin{equation*}
E^{\circ }:=\ \lim_{\lambda \uparrow \Lambda }(E\cap \lambda ).
\end{equation*}
\end{definition}

All the internal finite sets are hyperfinite, but there are hyperfinite sets
which are not finite. For example the set%
\begin{equation*}
\mathbb{R}^{\circ }:=\ \lim_{\lambda \uparrow \Lambda }(\mathbb{R}\cap
\lambda )
\end{equation*}%
is not finite. The hyperfinite sets are very important since they inherit
many properties of finite sets via Leibnitz principle. For example, $\mathbb{%
R}^{\circ }$ has the maximum and the minimum and every internal function%
\begin{equation*}
f:\mathbb{R}^{\circ }\rightarrow \mathbb{R}^{\ast }
\end{equation*}%
has the maximum and the minimum as well.

Also, it is possible to add the elements of an hyperfinite set of numbers or
vectors as follows: let%
\begin{equation*}
A:=\ \lim_{\lambda \uparrow \Lambda }A_{\lambda }
\end{equation*}%
be an hyperfinite set; then the hyperfinite sum is defined in the following
way: 
\begin{equation*}
\sum_{a\in A}a=\ \lim_{\lambda \uparrow \Lambda }\sum_{a\in A_{\lambda }}a.
\end{equation*}%
In particular, if $A_{\lambda }=\left\{ a_{1}(\lambda ),...,a_{\beta
(\lambda )}(\lambda )\right\} \ $with\ $\beta (\lambda )\in \mathbb{N},\ $%
then setting 
\begin{equation*}
\beta =\ \lim_{\lambda \uparrow \Lambda }\ \beta (\lambda )\in \mathbb{N}%
^{\ast }
\end{equation*}%
we use the notation%
\begin{equation*}
\sum_{j=1}^{\beta }a_{j}=\ \lim_{\lambda \uparrow \Lambda }\sum_{j=1}^{\beta
(\lambda )}a_{j}(\lambda ).
\end{equation*}


\begin{thebibliography}{99}
\bibitem{ultra} Benci V., \textsl{Ultrafunctions and generalized solutions,}
in: Advanced Nonlinear Studies, 13 (2013), 461--486, arXiv:1206.2257.

\bibitem{belu2012} Benci V., Luperi Baglini L., \emph{A model problem for
ultrafunctions}$,$ to appear in the EJDE Conference Proceedings,\emph{\ }%
arXiv:1212.1370.

\bibitem{belu2013} Benci V., Luperi Baglini L., \emph{Basic Properties of
ultrafunctions, }to appear in the WNDE2012 Conference Proceedings,
arXiv:1302.7156.

\bibitem{algebra} Benci V., Luperi Baglini L., \emph{An algebra of
ultrafunctions and distribution, }in preparation.

\bibitem{BDN2003a} Benci V., Di Nasso M. \textsl{Numerosities of labelled
sets: A new way of counting}, Adv. Math. \textbf{173} (2003), 50--67.

\bibitem{BDN2003} Benci V., Di Nasso M. \textsl{Alpha-theory: an
elementary axiomatic for nonstandard analysis}, Expositiones Mathematicae 
\textbf{21} (2003), 355--386.

\bibitem{BDNF2006} Benci V., Di Nasso M., Forti M., \textsl{An
Aristotelean notion of size}, Annals of Pure and Applied Logic \textbf{143}
(2006), 43--53.

\bibitem{benci99} Benci V., \textsl{An algebraic approach to nonstandard
analysis, }in: Calculus of Variations and Partial differential equations,
(G.Buttazzo, et al., eds.), Springer, Berlin, (1999), 285--326.

\bibitem{BGG} Benci V., Galatolo S., Ghimenti M., \textsl{An elementary
approach to Stochastic Differential Equations using the infinitesimals}, in
Contemporary Mathematics, \textbf{530}, Ultrafilters across Mathematics,
American Mathematical Society, (2010), 1--22.

\bibitem{BHW} Benci V., Horsten H., Wenmackers S., \textsl{Non-Archimedean
probability}, Milan J. Math., (2012), arXiv:1106.1524

\bibitem{DBR} Du Bois-Reymond P., \textsl{\"{U}ber die Paradoxen des Infinit%
\"{a}r-Calc\"{u}ls}, Math.Annalen \textbf{11}, (1877), 150--167.

\bibitem{el06} Ehrlich Ph., \textsl{The Rise of non-Archimedean Mathematics
and the Roots of a Misconception I: The Emergence of non-Archimedean Systems
of Magnitudes,} Arch. Hist. Exact Sci. \textbf{60}, (2006), 1--121.

\bibitem{LC} Levi-Civita T., \emph{Sugli infiniti ed infinitesimi attuali
quali elementi analitici},\ Atti del R. Istituto Veneto di Scienze Lettere
ed Arti, Venezia (Serie \textbf{7}), (1892--93), 1765--1815.

\bibitem{hilb} Hilbert D., \emph{Grundlagen der Geometrie}, Festschrift zur
Feier der Enth\"{u}llung des Gauss-Weber Denkmals in G\"{o}ttingen, Teubner,
Leipzig, (1899).

\bibitem{keisler76} Keisler H.J., \textsl{Foundations of Infinitesimal
Calculus}, Prindle, Weber \& Schmidt, Boston, (1976). [This book is now
freely downloadable at: http://www.math.wisc.edu/\symbol{126}%
keisler/foundations.html]

\bibitem{rob} Robinson A., \emph{Non-standard Analysis},\ Proceedings of
the Royal Academy of Sciences, Amsterdam (Series A) \textbf{64}, (1961),
432--440.

\bibitem{veronese} Veronese G., \textsl{Il continuo rettilineo e l'assioma
V di Archimede},\textquotedblright\ Memorie della Reale Accademia dei
Lincei, Atti della Classe di scienze naturali, fisiche e matematiche \textbf{%
4}, (1889), 603--624.

\bibitem{veronese2} Veronese G., \textsl{Intorno ad alcune osservazioni sui
segmenti infiniti o infinitesimi attuali},\textquotedblright\ Mathematische
Annalen \textbf{47}, (1896), 423--432.
\end{thebibliography}
\end{document}